\documentclass[article]{siamart0516}
\usepackage[utf8]{inputenc}

\usepackage{lipsum}
\usepackage{amsfonts}
\usepackage{graphicx}
\usepackage{epstopdf}
\usepackage{enumerate}
\usepackage{algorithmic}
\usepackage{color}
\usepackage{pdfpages}
\usepackage{nameref}
\usepackage{subfigure}
\usepackage{bbm}
\usepackage{amssymb}
\usepackage{amsmath}

\ifpdf
\DeclareGraphicsExtensions{-eps-converted-to.pdf,.pdf,.png,.jpg}
\else
\DeclareGraphicsExtensions{-eps-converted-to.pdf}
\fi

\hypersetup{urlcolor=blue, colorlinks=true}

\newcommand{\TheTitle}{Discretizations of the spectral fractional Laplacian on general domains with Dirichlet, Neumann, and Robin boundary conditions} 
\newcommand{\TheAuthors}{Nicole Cusimano, F\'elix del Teso, Luca Gerardo-Giorda, Gianni Pagnini}

\headers{Fractional Laplacian in bounded domains}{N. Cusimano, F. del Teso, L. Gerardo-Giorda, G. Pagnini}

\title{{\TheTitle}\thanks{Published in SIAM Journal on Numerical Analysis.}}

\author{
	Nicole Cusimano\thanks{Basque Centre for Applied Mathematics (BCAM), Bilbao, Spain (\email{ncusimano@bcamath.org}, \email{lgerardo@bcamath.org}, \email{gpagnini@bcamath.org})}
	\and
	F\'elix del Teso\thanks{Norwegian University for Science and Technology (NTNU), Trondheim, Norway (\email{felix.delteso@ntnu.no})}
	\and 
	Luca Gerardo-Giorda\footnotemark[2]
	\and 
	Gianni Pagnini\footnotemark[2]
		}

\usepackage{amsopn}

\newcommand{\B}{\mathcal{B}}
\newcommand{\R}{\mathbb{R}}

\newcommand{\N}{\mathbb{N}}
\newcommand{\Lit}{\mathcal{L}}

\newcommand{\LBs}{\mathcal{L}_{g}^s}

\newtheorem{remark}{Remark}
\newtheorem{defin}[theorem]{Definition}

\ifpdf
\hypersetup{
  pdftitle={\TheTitle},
  pdfauthor={\TheAuthors}
}
\fi

\numberwithin{equation}{section}
\numberwithin{theorem}{section}

\begin{document}

\maketitle

\begin{abstract}
In this work, we propose novel discretizations of the spectral fractional Laplacian on bounded domains based on the integral formulation of the operator via the heat-semigroup formalism. Specifically, we combine suitable quadrature formulas of the integral with a finite element method for the approximation of the solution of the corresponding heat equation. We derive two families of discretizations with order of convergence depending on the regularity of the domain and the function on which the spectral fractional Laplacian is acting. Our method does not require the computation of the eigenpairs of the Laplacian on the considered domain, can be implemented on possibly irregular bounded domains, and can naturally handle different types of boundary constraints. Various numerical simulations are provided to illustrate performance of the proposed method and support our theoretical results.	
\end{abstract}



\section{Introduction}
Fractional operators, and nonlocal operators in general,
became of high interest in mathematical modeling because of their success  
in overcoming difficulties encountered when trying to explain phenomena and fit data in complex systems.
Fractional generalizations of model equations have been studied both with time-fractional derivatives (e.g., diffusion problems in biologcal systems \cite{metzler_etal-pccp-2014,molina_etal-pre-2016}
and porous media \cite{plociniczak-cnsns-2015}, or in relaxation phenomena and viscoelasticity \cite{rossikhin_etal-amr-2010,mainardi-book-2010}) and with space-fractional operators (e.g., in plasma physics \cite{delcastillonegrete_etal-pp-2012}, quantum mechanics \cite{zhang_etal-prl-2015}, contaminant dispersion \cite{zhang_etal-wrr-2016} or cardiac electrical propagation~\cite{buenoorovio_etal-jrsi-2014,cusimano_etal-plos-2015})
in relation also to experimental evidence
(e.g.,~\cite{gambin_etal-prl-2005,golding_etal-prl-2006,ratynskaia_etal-prl-2006,mercadier_etal-np-2009,humphries_etal-pnas-2012,ariel_etal-nc-2015,manzo_etal-rpp-2015}).

Preliminary results of fractional models are promising and encourage the exploration of their use in more realistic settings, so that interesting mathematical models could phase into physically grounded equations, able to predict phenomena observed in experiments. 
The present study is motivated by providing accessible and effective tools for real world models involving fractional calculus operators.

In order to fix the physical foundation of fractional models,
bounded domains and boundary conditions come into play \cite{hilfer-fcaa-2015}.

Typically, experiments are set up so that no effects are expected from outside the domain where the sample is contained and the local boundary conditions should fully determine the evolution of the process under study. 
For example, the observable under investigation can be zero at the boundaries,
mathematically described by homogeneous Dirichlet boundary conditions.
However, depending on the phenomena and the domain observed, other boundary conditions involving the specification of a nonzero value of the observable or a particular expression of its flux at the boundaries are required. This last example is the case of an insulated domain where the imposition of reflecting boundaries (i.e., homogeneous Neumann boundary conditions) ensures mass conservation. 
When the observed process takes place in a domain characterized by high spatial complexity and/or structural heterogeneity we cannot ignore what happens at the interface between different structural compartments nor the conditions imposed at the boundary of the considered domain; hence the need for a consistent definition of nonlocal operators with boundary conditions. 

With this in mind, in the present paper we study a family of nonlocal operators defined 
in a consistent manner with the bounded domain on which they act and the corresponding local boundary conditions.

Since this issue concerns the nonlocality in space, we consider the fractional Laplacian as a prototype operator.  
For this nonlocal operator many definitions exist
(e.g., Fourier multiplier~\cite{SKM93}, Riesz potential operator~\cite{Lan72},
hypersingular integral~\cite{diN2012}, heat-semigroup~\cite{Sti2010}, 
extension problem~\cite{Caf2007}) 
and these definitions are proved to be equivalent for sufficiently smooth functions in $\mathbb{R}^n$, $n \geq 1$, \cite{Kw2015}.
However, when the nonlocal operator is considered on bounded domains, this equivalence no longer holds. In the literature, one can find many different approaches for the definition or restriction of the fractional Laplacian in the bounded domain case and often, for each of these definitions, various schemes are proposed for the numerical computation of the considered operator (see e.g., \cite{HO2014,Cia2015,dTe14,delia_etal-cma-2013,Bon2015,bonforte_etal-arxiv-2017, DrJa14,
	CiJa14,zayernuri_etal-siam-2014,DiP2010,Du2012,acosta_etal-siam-2015}). 

The work presented in this manuscript is based on the spectral definition of the fractional Laplacian on an open set $\Omega \subset \mathbb{R}^n$~\cite{cabre_etal-2010,stinga_etal-cpde-2010,Vaz2014}. Let $\Lit$ be a linear non-negative second order differential operator on $L^2(\Omega)$,  self-adjoint and with a compact inverse. Then there is an orthonormal basis of $L^2(\Omega)$ made of eigenfunctions $\{\varphi_m\}_{m\in\mathbb{N}}$ of $\Lit$, i.e., $\Lit \varphi_m=\lambda_m \varphi_m$, such that $\lambda_m \geq 0$, and $\lambda_m\to \infty$. Then, $\Lit u=\sum_m \lambda_m \langle u, \varphi_m \rangle  \varphi_m$, and fractional powers of the operator, namely $\Lit^s$, with $s\in (0,1)$, are defined as
\begin{equation}
\label{eq:spectraldef}
\Lit^s u= \sum_m \lambda_m^s \langle u, \varphi_m \rangle  \varphi_m,
\end{equation}
for all functions $u\in L^2(\Omega)$ such that $\sum_m \lambda_m^{2s} |\langle u, \varphi_m \rangle |^2 <\infty$ ($\langle \cdot , \cdot \rangle$ denotes the inner product in $L^2(\Omega)$). The choice of functional spaces on which $\Lit$ satisfies the aforementioned properties and the particular expression of eigenvalues $\lambda_j$ and eigenfunctions $\varphi_j$ depend not only on $\Lit$ but also on $\Omega$ and on the boundary conditions considered for this domain. Therefore, here and throughout the rest of the manuscript, we use the subscript $\mathcal{B}$ to highlight that a given operator, in the bounded setting, is considered together with some boundary conditions. For example, we denote the Laplacian on $\Omega$ by $(-\Delta_{\B})$ rather than simply $(-\Delta)$.

When $\Lit=(-\Delta_{\B})$, from \eqref{eq:spectraldef} we obtain the definition of the spectral fractional Laplacian $(-\Delta_{\B})^s$ on $\Omega$ and the notation is consistent with the one used for the Dirichlet and Neumann spectral fractional Laplacians in references \cite{Vaz2014}, \cite{stinga_etal-2015}, and \cite{grubb-mn-2016}. The fact that different boundary conditions modify the definition of the nonlocal operator on the bounded domain $\Omega$ is in complete agreement with our initial physical considerations and motivates our choice of working with the spectral definition rather than any other definition of fractional operator available in the literature.

Regarding the numerical approximation of fractional equations on bounded domains $\Omega$ involving the spectral fractional Laplacian, we mention here two strategies. On the one hand, the matrix transfer technique, initially proposed by
Ili\'c et al.~\cite{ITB2005,ITB2006}, computes a finite difference approximation of the spectral fractional Laplacian via the direct approximation of the eigenvalues and eigenvectors of the standard discrete Laplacian coupled to some given boundary conditions. 
This method can be easily adopted on simple geometries but, in the case of general domains, only estimates of the eigenpairs are available, and computations require the solution of a very large number of eigenvalue problems. An efficient method for approximating the eigendecomposition of the Laplacian on more general domains has been recently proposed in \cite{song_etal-2017}.
On the other hand, for general Lipschitz domains $\Omega\subset \mathbb{R}^n $, $n>1$, 
Nochetto et al.~\cite{Noc2015} developed a finite element (FE) technique to compute the solution 
to the fractional Poisson equation involving the spectral fractional Laplacian with homogeneous Dirichlet boundary conditions 
and interpreted in light of the Caffarelli--Silvestre extension method~\cite{Caf2007}. See \cite{dTe14} for a similar approach in the context of finite differences. Bonito and Pasciak in \cite{Bon2015} proposed an integral formulation for fractional powers of the inverse of the Laplacian to solve the fractional Poisson equation.

Differently from the approaches mentioned above, we exploit the fact that definition \eqref{eq:spectraldef} is equivalent to 
the one based on the heat-semigroup~\cite{Sti2010}. In particular, for the spectral fractional Laplacian we have
\begin{equation}
\label{eq:semigr}
(-\Delta_{\B})^s u(x) 
=\frac{1}{\Gamma (-s)}\int_0^\infty \left( e^{t\Delta_{\B}}u(x)-u(x)\right) \frac{dt}{t^{1+s}},
\end{equation}
where $e^{t\Delta_{\B}} \, u$ represents the solution of the classical heat equation on the bounded domain $\Omega$ with initial condition $u$ and boundary conditions $\B$.  

The connection between the nonlocal operator and the local heat equation provided by equation \eqref{eq:semigr} allows us to develop an accessible and effective method for the numerical computation of the considered nonlocal operator on general bounded domains and with various boundary constraints. Specifically, in the present paper we compute the value of the operator $(-\Delta_{\B})^s$ on a given function by combining a FE approach for the solution of the heat equation and a quadrature rule for the integral in~\eqref{eq:semigr} with respect to the measure $\displaystyle{d\mu(t)=\frac{dt}{t^{1+s}}}$. We provide a suitable discretization for the considered fractional operator and prove convergence of the numerical approximation. The cases of Dirichlet, Neumann and Robin boundary conditions are studied in details and several numerical results are provided. The numerical method is also used to compute the numerical approximation to a fractional porous medium type equation on a bounded domain: the sharp boundary behavior predicted in~\cite{bonforte_etal-arxiv-2017} is recovered.

Fractional powers of the Laplacian with homogeneous boundary conditions are our main focus. However, under suitable assumptions, we propose a possible strategy to deal with the case of non-homogeneous boundary conditions, and we outline how our approach can be applied to compute fractional powers of more general second order elliptic operator $\Lit$ and how it can be adapted to deal with the case of more general boundary conditions. 

\textbf{Organization of the paper.} In \cref{sec:main}, we introduce our framework, some definitions, the concepts of the FE theory, and provide the general form of the quadrature approximation that will be used in the rest of this work. After stating our main assumptions, in \cref{sec:LO} we provide our first discretization and the corresponding proof of convergence. Under suitable assumptions, a similar approach is used to derive a higher order discretization of the nonlocal operator, as shown in detail in \cref{sec:HO}. In \cref{sec:inhomBC} we study the case of non-homogeneous boundary conditions. We define a fractional-type operator corresponding to the Laplacian coupled with non-homogeneous boundary conditions and convergent discretizations of the considered operators are provided. Possible extensions of our main results to more general operators is outlined in \cref{sec:genL}.
In \cref{sec:experiments}, we present numerical experiments to support our results and to illustrate the performance of the presented numerical methods. Our conclusions and future research directions are given in \cref{sec:conclusions}.

\section{Definitions and concepts}
\label{sec:main}
\subsection{Notations and functional setting}
\underline{\textit{Boundary conditions.}} Let $\Omega$ be a bounded domain in $\mathbb{R}^n$, $u:\Omega \to \mathbb{R}$, $\kappa:\partial \Omega \to \mathbb{R}$, and $\frac{\partial u}{\partial \nu}$ be the derivative of $u$ along the unit outward normal direction on $\partial \Omega$. We denote by $\mathcal{B}$ one of the following boundary operators:  
\begin{align}
\label{eq:DirB}&\mathcal{B}(u)=u, \\
\label{eq:NeuB}&\mathcal{B}(u)=\frac{\partial u}{\partial \nu}, \\
\label{eq:RobB}&\mathcal{B}(u)=\kappa~u+\frac{\partial u}{\partial \nu}, \quad \kappa(x)>0 \ \text{ for all } \ x\in \partial \Omega,
\end{align}
which are known as Dirichlet, Neumann, and Robin boundary operators, respectively. For now we simply assume that the function $u$ is regular enough so that all of the above makes sense. 

\underline{\textit{Eigenvalue problem.}} Let $\Omega$ be a Lipschitz bounded domain and $\B$ one of the the above boundary operators. Then, there exists an orthonormal basis of $L^2(\Omega)$ consisting of eigenfunctions $\{\varphi_m\}_{m\in \N}$ of the Laplacian operator $-\Delta$ coupled to homogeneous boundary conditions $\B(\cdot)=0$, that is, non-trivial solutions of
	\begin{equation}
	\label{eq:eigProb}
	\left\{
	\begin{array}{lll}
	-\Delta \varphi_m=\lambda_m \varphi_m & \textup{in} &\Omega,\\
	\B(\varphi_m)=0 &  \textup{on} & \partial \Omega.
	\end{array} 
	\right.
	\end{equation}
	Moreover, the corresponding eigenvalues $\{\lambda_m\}_{m\in\mathbb{N}}$ are all non-negative and form a non-decreasing sequence such that $\lambda_m\to \infty$ (see e.g. \cite{Salsa2008}).

\underline{\textit{Functional spaces.}} Given an open set $\Omega$ and $r\geq 0$, $H^{r}(\Omega)$ denotes the classical Sobolev space $W^{r,2}(\Omega)$ (cf.~\cite{dNPaVa12}). 
We denote by $\|\cdot\|_r$ the classical norm of $H^r(\Omega)$. 
Throughout the manuscript we will always use $\|\cdot\|_0$ to indicate the norm in $H^0(\Omega)=L^2(\Omega)$. $H_0^1(\Omega)$ is the set of functions $u\in H^1(\Omega)$ with null trace on $\partial \Omega$.  
For any given $r\geq0$, let $\mathbb{H}_\B^r(\Omega)$ be the space of $u\in L^2(\Omega)$ such that
	\begin{equation}
	\label{eq:norm}
	\|u\|_{\mathbb{H}_\B^r}=\left(\sum_{m=1}^\infty\lambda_m^r|\hat{u}_m|^2\right)^{1/2}<\infty,
	\end{equation}
	where $\hat{u}_m:=\langle u, \varphi_m \rangle$ denotes the $L^2(\Omega)$ inner product of $u$ with the basis function $\varphi_m$ previously introduced. Most of the results in this manuscript will be stated in terms of the norm $\|\cdot\|_{\mathbb{H}_\B^r}$ given by \eqref{eq:norm}, for some suitable non-negative integer $r$. When $r$ is a non-negative integer,	
	$\mathbb{H}_\B^r(\Omega)$ has the following characterization (see Lemma~2.2 in \cite{bramble_etal-ampa-1974} for the Dirichlet case, with similar proofs for the Neumann and Robin cases): 
	\begin{equation}
	\label{eq:Hbchar}
	\mathbb{H}_\B^r(\Omega)=\{u\in H^r(\Omega) ~|~ \B(\Delta^ju)=0 \text{ on } \partial \Omega,  ~\forall~ \text{non-negative integer } j<r/2 \},
	\end{equation}and the norm $\| \cdot \|_{\mathbb{H}_\B^r}$ is equivalent to the usual norm $\|\cdot \|_r$ in $H^r(\Omega)$. The conditions imposed on $\Delta^j u$ at the boundary are known in the literature as \textit{compatibility conditions}.

\subsection{Fractional powers of the Laplacian in bounded domains} 
Fractional powers with exponent $s\in(0,1)$ of the Laplacian coupled to the boundary condition $\B$ are defined in the spectral sense as
\begin{equation}\label{eq:eigSerie}
(-\Delta_\B)^s u:= \sum_{m=1}^\infty \lambda_m^s \hat{u}_m \varphi_m,
\end{equation}
 for all functions $u\in L^2(\Omega)$ such that $\sum_m \lambda_m^{2s} |\hat{u}_m|^2<\infty$ (see e.g., \cite{cabre_etal-2010}). These conditions on $u$ are a quite common requirement in the literature dealing with the spectral fractional Laplacian and we will proceed under these assumptions. However, wanting to work in full generality, the assumptions on the function $u$ could be relaxed.
 
Let $\B$ be one of the boundary operators introduced in \eqref{eq:DirB}, \eqref{eq:NeuB}, or \eqref{eq:RobB}. We let $w(x,t):=e^{t\Delta_{\B} }u(x)$ denote the solution of the heat equation
	\begin{equation}
	\label{eq:heat}
	\left\{
	\begin{array}{lcl}
	\partial_t w-\Delta w = 0 & \quad & (x,t)\in \Omega\times (0,\infty),\\
	w(x,0)=u(x)& & x\in \Omega,\\
	\B(w(\cdot,t))(x)=0 & & (x,t) \in \partial \Omega\times [0,\infty).
	\end{array}
	\right.
	\end{equation}
	By separation of variables, the function $w$ can be written as the following series:
	\begin{equation}
	\label{eq:Wseries}
	w(x,t)=\sum_{m=1}^\infty e^{-\lambda_m t}\hat{u}_m \varphi_m(x),
	\end{equation}
	where $\lambda_m$ and $\varphi_m$ are the eigenpairs defined in \eqref{eq:eigProb}. Let $w_\infty$ denote the steady-state of \eqref{eq:general}. Using the series expansion of $w$ we can easily see that $w_\infty=0$ when $\lambda_m>0$ for all $m$ (as in the Dirichlet case and the Robin case with $\kappa>0$), while $w_\infty=\hat{u}_1\varphi_1=\frac{1}{|\Omega|}\int_{\Omega}u~dx$ when $\lambda_1=0$ and $\lambda_m>0$ for $m\geq 2$ (as in the Neumann case).

On a bounded domain $\Omega$, the definition of $(-\Delta_\B)^s$ given by \eqref{eq:eigProb}-\eqref{eq:eigSerie} is equivalent to the following one:
\begin{defin}\label{def:fraclap}
Let $\Omega$ be a bounded domain, $s\in(0,1)$ and $u\in \mathbb{H}_\B^{2s}(\Omega)$. The \textbf{spectral $\B$-fractional Laplacian} of $u$ in $\Omega$ is given by
\begin{equation}
\label{eq:start}
(-\Delta_{\B})^{s}u(x)=\frac{1}{\Gamma(-s)}~\int_0^\infty \left(e^{t\Delta_{\B} }u(x)-u(x)\right)~ \frac{dt}{t^{1+s}},
\end{equation}
where $e^{t\Delta_{\B} }u(x)$ is the solution of \eqref{eq:heat}.
\end{defin}
In \cref{sec:inhomBC} we will extend definition \eqref{eq:start} to a fractional-type operator associated with non-homogeneous boundary conditions.

\subsection{Basics on the finite elements theory}
To compute FE approximations of the heat equation solution, we consider the heat problem \eqref{eq:heat} in weak form.
The weak formulation of the heat equation~\eqref{eq:heat} is as follows: given $u\in L^2(\Omega)$ find $w\in L^2(0,\infty; V)\cap C([0,\infty);L^2(\Omega))$ such that
\begin{equation}
\label{eq:general}
\left\{
\begin{array}{l}
\dfrac{d}{dt} \langle w(t),v\rangle+a( w(t),v)=0, \quad \forall v \in V, t>0\\
w(0)=u,
\end{array}
\right.
\end{equation}
where $a(\cdot,\cdot)$ is the bilinear form corresponding to the particular set of boundary conditions considered in~\eqref{eq:heat} and $V$ is a suitable functional space (i.e., $H^1_0(\Omega)$ for homogeneous Dirichlet conditions or $H^1(\Omega)$ in the homogeneous Neumann and Robin cases). In particular,
$a(w,v)=\int_{\Omega}\nabla w \cdot \nabla v$ for the Dirichlet and Neumann cases, while $a(w,v)=\int_{\Omega}\nabla w \cdot \nabla v + \langle \kappa~ w,v\rangle_{L^2(\partial \Omega)}$ in the Robin case, with $\langle \cdot,\cdot\rangle_{L^2(\partial \Omega)}$ denoting the scalar product in $L^2(\partial \Omega)$. Note that the series expansion given by \eqref{eq:Wseries} still holds for the weak solution $w$ (see e.g. Lemma 7.2-1 in \cite{RT1983}).

We introduce here the notation and concepts of the FE theory that we will use throughout this manuscript. 

	\noindent \textbf{Triangulation}. 
	Let $h>0$ and $\mathcal{T}_h$ denote a triangulation  of $\bar{\Omega}$ such that the diameter of each element $K \in\mathcal{T}_h$ is smaller than $h$. 
	
	\noindent \textbf{Spatial nodes}. Let $\Omega_h$ denote the set $\{x_i\}_{i=1}^{N_h}$ of $N_h$ nodes in the considered triangulation.
	
	\noindent \textbf{Time points}. Let $\varDelta t>0$ and $t_j=j\varDelta t$, for $j=0,1,\dots,N_t$, be a uniform grid of time points used to discretize the interval $[0,T]$, for a given $T>0$. 
	
	\noindent \textbf{Finite dimensional space and FE basis}. Let $\mathbb{P}_k$, $k\geq 1$, denote the space of polynomials of degree less than or equal to $k$ and let $X_h^k$ be the space of triangular FE, that is, 
	\[
	X_h^k:=\left\{v_h\in C^0(\bar{\Omega})~|~ v_{h|K}\in \mathbb{P}_k ~~\forall K\in \mathcal{T}_h \right\}.
	\]
	Let $V_h=V\cap X_h^k$ be the finite dimensional space approximating $V$ and $\{\phi_i\}_{i=1}^{N_h}$ be the FE basis of $V_h$ given by the functions $\phi_i\in X_h^k$, such that $\phi_i(x_j)=\delta_{ij}$ for all $x_j\in\Omega_h$.

	\noindent \textbf{$L^2(\Omega)$-orthogonal projection}. Let $u\in L^2(\Omega)$ and $P^k_h$ be the $L^2(\Omega)$-orthogonal projection operator over the space $X_h^k$, that is, the operator $P^k_h:L^2(\Omega)\to X^k_h$ such that $\langle P^k_hu,\phi\rangle = \langle u, \phi \rangle$ for all $\phi\in X_h^k$.

	\noindent \textbf{Stiffness and Mass matrices}. The stiffness matrix $A$ and the mass matrix $M$ corresponding to the FE approximation of the problem \eqref{eq:general} on $V_h$ have entries $A_{ij}=a(\phi_i,\phi_j)$ and $M_{ij}=\langle \phi_i,\phi_j\rangle$ for $i,j=1,\dots,N_h$, respectively.

\begin{defin}\label{def:FE}
Let $\Omega$ be a bounded domain, $u\in L^2(\Omega)$, $h,\varDelta t>0$, and $\theta\in[0,1]$. Let $\mathbf{W}^{(j)}$ be the solution vector of the iterative linear system
\begin{equation}
	\label{eq:fullydisc}
	\begin{array}{l}
	(M+\theta \varDelta t A)\mathbf{W}^{(j)}=(M+(\theta-1)\varDelta t A)\mathbf{W}^{(j-1)}\quad \text{for}\quad j=1,\ldots,N_t, \quad 
	\end{array}
\end{equation}
where the initial condition $\mathbf{W}^{(0)}$ has components $W_i^{(0)}=P^k_hu(x_i)$, for $i=1,\dots,N_h$.
The \textbf{FE solution of degree $k$} of \eqref{eq:general} is the function $w_h:\overline{\Omega} \times  \{(\varDelta t\cdot \mathbb{N})\cap [0,T]\}\to \mathbb{R}$ such that
	\begin{equation}
	\label{eq:FEsol}
	w_h(x,t_j)=\sum_{i=1}^{N_h}W_i^{(j)}\phi_i(x), \quad \text{for all}\quad x\in \overline{\Omega},
	\end{equation}
	where $\{\phi_i\}_{i=1}^{N_h}$ is the FE basis of degree $k$.
\end{defin}

\subsection{Definition of the discrete operator}
Given $u\in \mathbb{H}_\B^{2s}(\Omega)$, we denote by $\Theta^s$ the following discrete operator in time:
\begin{equation}\label{eq:quadrule}
	\Theta^s u(x):=\frac{1}{\Gamma(-s)}\left[\sum_{j=1}^{N_t}\left( w(x,t_j)-u(x)\right)~\beta_j+(w_\infty -u(x))~\beta_\infty\right],
\end{equation}
where $w(x,t_j)$ is the solution of \eqref{eq:general} at the time points $t_j=j\varDelta t$, $N_t\in \mathbb{N}$ is a suitable integer, the $\beta_j$ are weights defining a suitable quadrature rule for \eqref{eq:start}, $w_\infty$ is the steady-state of \eqref{eq:general}, and $\beta_\infty$ is the value of the tail of the integral with respect to the measure $\frac{dt}{t^{1+s}}$(see \cref{thm:main} and \cref{thm:HOmain} for further details).

Combining \eqref{eq:quadrule} with the FE we get the following approximation of the spectral $\B$-fractional Laplacian $(-\Delta_{\B})^{s}$ of a given $u$:

\begin{defin}\label{def:discOp}
	Let $w_h$ be the FE solution of degree $k$ of \eqref{eq:general} defined by equation \eqref{eq:FEsol}, corresponding to the initial datum $u$. Then, for all $x\in\Omega$, the \textbf{approximation} of the $\B$-fractional Laplacian of $u$ is defined as 
	\begin{equation}
	\label{eq:discop}
	\Theta^s_h u(x) :=\frac{1}{\Gamma(-s)}\left[\sum_{j=1}^{N_t}(w_h(x,t_j)-w_h(x,0))~\beta_j +(w_\infty-u(x))~\beta_\infty\right].
	\end{equation}
\end{defin}

\begin{remark}
In \cref{sec:LO} and \cref{sec:HO} we propose two different choices of the weights $\beta_j$ leading to approximations of $(-\Delta_\B)^s u$ with different accuracy. The weights $\beta_j$ defined in \eqref{eq:weights1} correspond to a composite midpoint quadrature rule while the ones in \eqref{eq:defBetas} correspond to a quadrature defined via a piecewise linear interpolant.
\end{remark}

\section{Main result}
\label{sec:LO}
\subsection{Assumptions}
In order to prove our main convergence results, we need some assumptions on the domain $\Omega$, the triangulation $\mathcal{T}_h$, the function $u$, the spatial and the temporal steps $h$ and $\Delta t$, respectively. Specifically, we assume:
\begin{enumerate}
	\item[($A_\Omega$)] $\Omega$ is a bounded convex polytope (see Remark 2 below).
	
	\item[($A_{\B}$)] The boundary operator $\B$ is given by \eqref{eq:DirB}, \eqref{eq:NeuB}, or \eqref{eq:RobB}. If $\B$ is the Robin boundary operator, then $\kappa\in C^1(\partial \Omega)$ and $\kappa(x)>0$ for all $x\in \partial \Omega$.
	
	\item[($A_u$)] $u\in \mathbb{H}_\B^{k+1}(\Omega)$, with $k\geq 1$ (degree of the FE).
	
	\item[($A_{\mathcal{T}_h}$)] $\mathcal{T}_h$ is a quasi-uniform family of triangulations: letting $h_K:= \mathrm{diam}(K)$, $K\in\mathcal{T}_h$, and 
	$\rho_K:= \mathrm{sup}\{\mathrm{diam}(S)|S\text{ is a ball contained in }K\}$, 
	there exist two constants, $\sigma\geq 1$ and $\tau>0$, such that $\max_K h_K/\rho_K \leq \sigma$ and $\min_K h_K\geq \tau~ h$, $\forall h>0$.

	\item[($A_{\text{CFL}}$)] If $0\leq \theta <\frac{1}{2}$, then $\varDelta t ~h^{-2}\leq c_\mu/(1-2\theta)$,	where $c_\mu$ is a positive constant independent from both $\varDelta t$ and $h$, derived from a bound for the largest eigenvalue of the bilinear form $a(\cdot,\cdot)$ on the finite-dimensional space $V_h$ (see e.g., Chapter 6 in reference~\cite{QV2008} for details).
\end{enumerate}

\begin{remark}
		We have assumed the domain to be a convex polytope so that the technicalities related to FE approximations in the case of curvilinear boundaries and re-entrant corners can be avoided. However, assumption \textup{($A_\Omega$)} can be relaxed to $\Omega$ being a bounded domain with Lipschitz boundary $\partial \Omega$. The case of Lipschitz domains with curvilinear boundaries can be dealt with, for example, as proposed in \cite{bernardi-siam-1989}, while suitable mesh refinements near re-entrant corners allow the recovery of optimal error bounds for non-convex domains as shown in the work by Chatzipantelidis et al. \cite{chatzipantelidis_etal-bit-2006} and references therein.
		\end{remark}

\subsection{Low order approximation}

We now state our first main result, which concerns homogeneous boundary conditions.

\begin{theorem}\label{thm:main}
Let $s\in(0,1)$, $h>0$, $\theta\in[0,1]$, $k\in\mathbb{N}$, and 
\begin{equation}\label{eq:restr}
\varDelta t = \eta~ h^p \quad \text{for some } \eta>0 \text{ and } p \in\left\{ 
\begin{split}
[2,k+1] \quad \text{if } \quad 0\leq \theta <\frac{1}{2},  \\
(0,k+1] \quad \text{if } \quad \frac{1}{2}\leq \theta \leq1.
\end{split} \right.
\end{equation}
Assume that \textup{($A_\Omega$), ($A_u$), ($A_{\B}$), ($A_{\mathcal{T}_h}$), ($A_{\text{CFL}}$)} hold, and let $\Theta^s_h  $ be defined by \eqref{eq:discop} where:
\begin{enumerate}[(a)]
\item $w_h$ is the FE solution of degree $k$ of \eqref{eq:general} with initial datum $u$;

\item $N_t\geq \dfrac{1-s}{\lambda_{\rm{min}}\varDelta t}\log\left(\dfrac{1}{\varDelta t}\right)$, $\lambda_{\rm{min}}$ being the first nonzero eigenvalue of $(-\Delta_\B)$;

\item the weights $\beta_j$, 
for $j=1,\dots,N_t$, are given by 
\begin{equation}\label{eq:weights1}
\beta_j:=\int_{t_j-\varDelta t/2}^{t_j+\varDelta t/2}\frac{dt}{t^{1+s}}=\frac{1}{s}\left[ \frac{1}{\left(t_j-\frac{\varDelta t}{2}\right)^s}-\frac{1}{\left(t_j+\frac{\varDelta t}{2}\right)^s}\right],
\end{equation}
and 
\[
\beta_\infty:=\int_{t_{N_t}+\varDelta t/2}^\infty \frac{dt}{t^{1+s}}= \frac{1}{s ((N_t+1/2)\varDelta t)^s}.
\]

\end{enumerate}
Then, there exists a constant $C>0$ independent of $h$ such that
\begin{equation}
\label{eq:ofMain}
	\|(-\Delta_{\mathcal{B}})^su-\Theta^s_h u\|_{0} \leq C~h^{p(1-s)}.
\end{equation}
\end{theorem}

\subsection{Proof of \cref{thm:main}}
\label{sec:proof}
In order to prove \cref{thm:main} we need some intermediate results. 
We split the integral in \eqref{eq:start} into three parts:
\begin{enumerate}[(i)]
	\item \textbf{singular part}
	\begin{equation}\label{eq:singPart}
	I_S [u](x):=\frac{1}{\Gamma(-s)}\int_0^{T_1} \left(e^{t\Delta_\B }u(x)-u(x)\right) ~ \frac{dt}{t^{1+s}},  \qquad \textup{for some}\qquad  T_1>0,
	\end{equation}
	\item \textbf{middle part}
	\begin{equation}\label{eq:midPart}
	I_M [u](x):=\frac{1}{\Gamma(-s)}\int_{T_1}^{T_2} \left(e^{t\Delta_\B }u(x)-u(x)\right) ~ \frac{dt}{t^{1+s}}, \qquad \textup{for some} \qquad  T_2>T_1,
	\end{equation}
	\item \textbf{tail}  
	\begin{equation}\label{eq:tailPart}
	I_\infty [u](x):=\frac{1}{\Gamma(-s)}\int_{T_2}^\infty \left(e^{t\Delta_\B }u(x)-u(x)\right) ~ \frac{dt}{t^{1+s}}.
	\end{equation}
\end{enumerate}
This way, $(-\Delta_{\B})^su=I_S [u]+I_M[u]+I_\infty[u]$. For each of these integrals, in \cref{sec:quad}, we define an approximation, that we denote by $I_S^D, I_M^D, I_\infty^D$, so that
 \[
\Theta^s u=I_S^D [u]+I_M^D[u]+I_\infty^D[u].
 \]
In \cref{sec:quad} we provide a bound to $\|(-\Delta_\B)^su- \Theta^s u\|_0$ while in \cref{sec:numApp} we estimate $\|\Theta^s u - \Theta^s_h u \|_0$. The proof of \cref{thm:main} follows from these preliminary results.

\subsubsection{Intermediate results}
In the following lemmas, we work under the assumptions \textup{$(A_\Omega)$, $(A_\B)$,} and that the regularity of $u$ is as given explicitly by the constants in each error estimate.	
\label{sec:quad}
\begin{lemma}
	\label{lem:S}
	Let $s\in(0,1)$, $\varDelta t>0$, $I_S$ given by \eqref{eq:singPart} with $T_1:=\varDelta t/2$, and $I_S^D[u]:=0$. Then, there exists a constant $c_S=c_S(s,\| u\|_{\mathbb{H}_\B^2})>0$ such that
	\[
	\|I_S[u]-I_S^D[u]\|_0\leq c_S~\varDelta t^{1-s}.
	\]
	\end{lemma}
	\begin{proof}
Using the series expansion of $u$ and $w=e^{t\Delta_\B}u$ (see \eqref{eq:Wseries}), 
\[
w(x,t)-u(x)=\sum_{m} (e^{-\lambda_m t}-1)~ \hat{u}_m~ \varphi_m(x).
\]
By orthonormality of the functions $\varphi_m$, 
\[
\|w(\cdot,t)-u\|_0=\left[\sum_m \left( e^{-\lambda_m t}-1\right)^2 |\hat{u}_m|^2\right]^{1/2}.
\]
Moreover, for all $ \lambda\geq 0$ and $t\geq 0$, $(e^{-\lambda t}-1)^2 \leq \lambda^2t^2$, so that
\[
\|w(\cdot,t)-u\|_0\leq \left[\sum_m \lambda_m^2t^2 |\hat{u}_m|^2\right]^{1/2}=\| u\|_{\mathbb{H}_\B^2}t.
\]
Therefore, exploiting Minkowski's integral inequality
\begin{align*}
\|I_S[u]-I_S^D[u]\|_0	&\leq \frac{1}{|\Gamma(-s)|}\int_0^{\varDelta t/2}\|w(\cdot,t)-u \|_0 \frac{dt}{t^{s+1}}\\
&\leq \frac{\| u\|_{\mathbb{H}_\B^2}}{|\Gamma(-s)|}\int_0^{\varDelta t/2}  \frac{dt}{t^{s}}\leq c_S ~\varDelta t^{1-s}.
\end{align*}
\end{proof}

\begin{lemma}
\label{lem:M}
Let $s\in(0,1)$, $\varDelta t>0$, $I_M$ given by \eqref{eq:midPart} with $T_1:=\varDelta t/2$ and $T_2=(N_t+1/2)\varDelta t$ for some positive integer $N_t$. Let $t_j=j\varDelta t$ for $j=1,\dots,N_t$ and  
	\begin{equation}
	\label{eq:midlowDisc}
	I_M^D[u](x):=\frac{1}{\Gamma(-s)}\sum_{j=1}^{N_t}(e^{t_j\Delta_\B}u(x)-u(x))\beta_j,
	\end{equation}
	 with weights $\beta_j$ given in \eqref{eq:weights1}. Then, there exists a constant $c_M=c_M(s,\|u\|_{\mathbb{H}_\B^2})>0$ such that
	\[
	\|I_M[u]-I_M^D[u]\|_0\leq c_M~\varDelta t^{1-s}.
	\]
\end{lemma}

\begin{proof}
We split the interval $[\varDelta t/2,(N_t+1/2)\varDelta t]$ into $N_t$ subintervals of size $\varDelta t$, denote $w=e^{t\Delta_\B}u$, and rewrite $I_M[u]$ as the following finite sum of integrals:
\[
I_M[u](x)=\frac{1}{\Gamma(-s)}\sum_{j=1}^{N_t} \int_{t_j-\varDelta t/2}^{t_j+\varDelta t/2} (w(x,t)-u(x))~\frac{dt}{t^{1+s}},
\]
where $t_j=j~\varDelta t$, for $j=1,\dots,N_t$. By using the expression \eqref{eq:midlowDisc} for $I_M^D[u]$ we obtain
\[
I_M[u](x)-I_M^D[u](x)=\frac{1}{\Gamma(-s)}\sum_{j=1}^{N_t} \int_{t_j-\varDelta t/2}^{t_j+\varDelta t/2} (w(x,t)-w(x,t_j))~\frac{dt}{t^{1+s}},
\]
so that, by using the triangular inequality and Minkowski's integral inequality,  
\[
\|I_M[u]-I_M^D[u]\|_0\leq\frac{1}{|\Gamma(-s)|}\sum_{j=1}^{N_t} \int_{t_j-\varDelta t/2}^{t_j+\varDelta t/2} \|w(\cdot,t)-w(\cdot,t_j)\|_0~\frac{dt}{t^{1+s}}.
\]
Using the series expansion of $w$ and the fact that for all $\lambda\geq 0$, $e^{-\lambda t}$ is a convex and monotonically decreasing function of $t$, one finds that on each subinterval $\mathcal{I}_j=[t_j-\varDelta t/2,t_j+\varDelta t/2]$,
\[
(e^{-\lambda_m t}-e^{-\lambda_m t_j})^2\leq \lambda_m^2 e^{-2\lambda_m(t_j-\varDelta t/2)}(t-t_j)^2 \leq \lambda_m^2 \varDelta t^2.
\] 
Therefore, $\|w(\cdot,t)-w(\cdot,t_j)\|_0\leq \varDelta t \| u\|_{\mathbb{H}_\B^2}$ for all $j$ and all $t\in \mathcal{I}_j$, which leads to 
 \begin{align*}
\|I_M[u]-I_M^D[u]\|_0& \leq \frac{\|u\|_{\mathbb{H}_\B^2}~\varDelta t}{|\Gamma(-s)|} \sum_{j=1}^{N_t} \int_{t_j-\varDelta t/2}^{t_j+\varDelta t/2}\frac{dt}{t^{1+s}}=\frac{\| u\|_{\mathbb{H}_\B^2}~\varDelta t}{|\Gamma(-s)|}\int_{\varDelta t/2}^{(N_t+1/2)\varDelta t} \frac{dt}{t^{1+s}}\\
& \leq \frac{\|u\|_{\mathbb{H}_\B^2}~ \varDelta t}{|\Gamma(-s)|} ~\frac{2^s}{s\varDelta t^s}\leq c_M ~\varDelta t^{1-s}.
\end{align*}
\end{proof}

\begin{lemma}
		\label{lem:inf}
		Let $s\in(0,1)$, $I_\infty$ given by \eqref{eq:tailPart} for some $T_2>0$, and
		\begin{equation}
		\label{eq:taildisc}
		I_\infty^D[u](x):=\frac{(w_\infty-u(x))}{\Gamma(-s)} \int_{T_2}^\infty \frac{dt}{t^{1+s}}=\frac{(w_\infty -u(x))}{\Gamma(-s)~s~T_2^s}.
		\end{equation}
		Let $\lambda_{\rm{min}}$ be the first nonzero eigenvalue of $(-\Delta_\B)$. Then there exists a constant $c_\infty=c_\infty(s,\|u\|_0)$ such that 
		\[
		\| I_\infty[u]-I_\infty^D[u]\|_0\leq c_\infty \frac{e^{-\lambda_{\rm{min}}T_2}}{T_2^{s}}.
		\]
	\end{lemma}
	\begin{proof}
		Recall that $w_\infty=0$ for the Dirichlet and Robin cases, while $w_\infty=\hat{u}_1 \varphi_1$ for Neumann boundary conditions, and let $m_{\rm{min}}$ be the index corresponding to $\lambda_{\rm{min}}$ (i.e., for the boundary operators of interest, $m_{\rm{min}}=1$ in the Dirichlet and Robin cases, while $m_{\rm{min}}=2$ in the Neumann case). Then, for all $t\geq 0$, $w(x,t)-w_\infty=\sum_{m=m_{\rm{min}}}^\infty e^{-\lambda_m t}\hat{u}_m \varphi_m(x)$, so that
		\[
		\|w(\cdot,t)-w_\infty\|_0=\left[ \sum_{m=m_{\rm{min}}}^{\infty}e^{-2\lambda_m t}|\hat{u}_m|^2 \right]^{1/2}\leq e^{-\lambda_{\rm{min}}t} \|u\|_0.
		\]
		Hence,
		\[
		\label{eq:quadInfBetter}
		\|I_\infty[u]-I_\infty^D[u]\|_0\leq \frac{1}{|\Gamma(-s)|}\int_{T_2}^\infty \|w(\cdot,t)-w_\infty\|_0\frac{dt}{t^{1+s}}\leq \frac{\|u\|_0~ e^{-\lambda_{\rm{min}}T_2 }}{|\Gamma(-s)|sT_2^s}.
		\]
	\end{proof}

\begin{proposition}
	\label{prop:Quad}
	Let assumptions of \cref{thm:main} hold. Let $(-\Delta_\B)^s$ and $\Theta^s $ be defined by \eqref{eq:start} and \eqref{eq:quadrule}, respectively. Then, there exists a constant $c_a=c_a(s,\|u\|_{\mathbb{H}_\B^2})>0$ such that
	\[
	\| (-\Delta_\B)^su-\Theta^s u\|_0\leq c_a~\varDelta t^{1-s}.
	\]
	\end{proposition}
	\begin{proof}
		The result follows by simply combining \cref{lem:S}, \cref{lem:M}, and \cref{lem:inf}. 	
		\begin{equation*}
			\begin{split}
				\| (-\Delta_\B)^su-\Theta^s u\|_0&\leq \|I_S[u]-I_S^D[u]\|_0+\|I_M[u]-I_M^D[u]\|_0+\|I_\infty [u]-I_\infty^D[u]\|_0\\
				&\leq c_S~\varDelta t^{1-s}+c_M~\varDelta t^{1-s}+c_\infty\frac{e^{-\lambda_{\text{min}}T_2}}{T_2^s}.
			\end{split}
		\end{equation*}
		Finally, using the definition of $N_t$ and choosing $T_2=(N_t+1/2)\varDelta t$ we obtain that $\frac{e^{-\lambda_{\text{min}}T_2}}{T_2^{s}}\leq \varDelta t^{1-s}$, and a suitable choice of the constant $c_a$ completes the proof.
		\end{proof}

\subsubsection{Finite element approximation}
\label{sec:numApp}
In this section we combine some well-known estimates of the FE method to obtain a bound for the error made by approximating $\Theta^su$ by $\Theta^s_h u$. As usual, we denote by $w$ the solution of the heat equation with initial condition $u$, and by $w_h$ the corresponding FE approximation of order $k$. 

Recalling that $w(\cdot,0)=u$ and that $w_h(\cdot,0)=P^k_h u$ is the $L^2(\Omega)$-orthogonal projection of the function $u$, from classical estimates for the FE projection error (e.g., Section 3.5 in \cite{QV2008}) we have
\begin{equation}
\label{eq:estTri}
\|w(\cdot,0)-w_h(\cdot,0)\|_0=\|u-P^k_hu\|_0\leq c~ h^{k+1}||u||_{k+1}
\end{equation}
for some positive constant $c$ independent from $h$. 

In order to bound $\|w(\cdot,t_j)-w_h(\cdot,t_j)\|_0$ for $t_j>0$ we use the fact that there exists a constant $\tilde{c}$ independent from $h$, $\varDelta t$, and $t_j$ such that
\begin{equation}
\label{eq:FElowOrd}
\|w(\cdot,t_j)-w_h(\cdot,t_j)\|_0\leq \tilde{c}~ (h^{k+1}+\varDelta t).
\end{equation}
The above result, for $t_j>1$ can be obtained via classical FE estimates which hold even for rough initial data, $u\in L^2(\Omega) $ (see e.g., Remark 11.3.2 in \cite{QV2008} or Theorems 1 and 2 in \cite{luskin_etal-siam-1981}), while for $t_j\leq 1$, the bound \eqref{eq:FElowOrd} is a consequence of the regularity assumptions made on $u$, ensuring that the constant $\tilde{c}$ can be controlled by the integral (over $[0,1]$) of suitable spatial norms of the heat solution $w$ and its temporal derivatives (see e.g., Theorem 11.3.4 in \cite{QV2008} or Theorem 2.1 in \cite{bramble_etal-siam-1977}).

\begin{proposition}
	\label{prop:M}
Let assumptions of \cref{thm:main} hold and $\Theta^s u$ given by \eqref{eq:quadrule}. Then, there is a constant $c_D=c_D(s,\Omega,\|u\|_{k+1},\B,k)>0$ such that
\[
\|\Theta^s u- \Theta^s_h u \|_0\leq c_D~( h^{k+1} \varDelta t^{-s} + \varDelta t^{1-s}).
\]
\end{proposition}

\begin{proof}
Via the triangular inequality we find
\begin{equation}
	\label{eq:bound0}
\|\Theta^s u - \Theta^s_h u \|_0 \leq \frac{1}{|\Gamma(-s)|}\sum_{j=1}^{N_t} \Big[\|w(\cdot,t_j)-w_h(\cdot,t_j)\|_0 
+\|w(\cdot,0)-w_h(\cdot,0)\|_0 \Big]\beta_j.
\end{equation}
On one hand, by using \eqref{eq:estTri}, we can bound the sum of terms in \eqref{eq:bound0} involving $w(\cdot,0)$ and $w_h(\cdot,0)$ as follows:
\begin{equation*}\label{eq:estIn}
\begin{split}
\frac{1}{|\Gamma(-s)|}\sum_{j=1}^{N_t} \|w(\cdot,0)-w_h(\cdot,0)\|_0 ~\beta_j&\leq \frac{c~h^{k+1}~\|u\|_{k+1}}{|\Gamma(-s)|}\sum_{j=1}^{N_t}\beta_j\\
&= \frac{c~h^{k+1}~\|u\|_{k+1}}{|\Gamma(-s)|}\int_{\varDelta t /2}^{T_2} \frac{dt}{t^{1+s}} \leq c_0~ h^{k+1} \varDelta t^{-s}.
\end{split}
\end{equation*}
On the other hand, for the remaining terms of \eqref{eq:bound0} we use \eqref{eq:FElowOrd} and obtain
\begin{equation*}
\begin{split}
		\frac{1}{|\Gamma(-s)|}\sum_{j=1}^{N_t}\|w(\cdot,t_j)-w_h(\cdot,t_j)\|_0 ~\beta_j  &\leq \frac{1}{|\Gamma(-s)|} \tilde{c}~\left(h^{k+1}+\varDelta t\right)  \int_{\varDelta t/2}^{T_2} \frac{dt}{t^{1+s}} \\
		&\leq  c_1\left(h^{k+1}+\varDelta t\right) \varDelta t^{-s}.
		\end{split}
\end{equation*}
Hence, combining the above estimates we get
\begin{align*}
\|\Theta^s u - \Theta^s_h u \| &\leq c_0~ h^{k+1} \varDelta t^{-s}+ c_1 \left(h^{k+1}+\varDelta t\right) \varDelta t^{-s}\\
&\leq c_D \left(h^{k+1}\varDelta t^{-s} + \varDelta t^{1-s} \right).
\end{align*}
\end{proof}

\subsubsection{Proof of the main result}
\begin{proof}[Proof of \cref{thm:main}]
		The triangular inequality yields
		\begin{equation}
		\label{eq:triang1}
		\|(-\Delta_\B)^su-\Theta^s_h u\|_0\leq \|(-\Delta_\B)^su-\Theta^s u\|_0
		+\|\Theta^s u-\Theta^s_h u\|_0.
		\end{equation}
\cref{prop:Quad} provides the following bound of the first term in \eqref{eq:triang1}:
\[
\|(-\Delta_\B)^su-\Theta^s u\|_0 \leq c_a ~\varDelta t^{1-s}.
\]
On the other hand, from \cref{prop:M} we have that 
		\[
\|\Theta^s u-\Theta^s_h u\|_0 \leq c_D( h^{k+1} \varDelta t^{-s} + \varDelta t^{1-s}).
		\]
From \eqref{eq:restr} we have that  $h^{k+1}\varDelta t^{-s}=\eta^{-s}h^{k+1-ps}$ and $\varDelta t^{1-s}=\eta^{1-s}h^{p(1-s)}$. Moreover, $k+1-ps\geq p(1-s)$ for all $p\leq k+1$. Hence, the order of convergence is always determined by $h^{p(1-s)}$, and by suitably choosing the constant $C$ in \eqref{eq:ofMain} we conclude the proof.
\end{proof}
		
\section{Higher order discretization}
\label{sec:HO}

Since $s\in(0,1)$, the  rate of convergence presented in \cref{thm:main} can be slow for practical computations when $s$ is close to one. In this section, we present a discretization that allows to obtain a better rate of convergence at the expenses of requiring some extra regularity to the function $u$, and also to the boundary operator in the Robin case. We adopt the Crank-Nicolson temporal scheme, corresponding to setting $\theta=1/2$ in the FE solution of the heat equation given by \eqref{eq:fullydisc}.
We suitably modify the relationship between the discretization parameters $\varDelta t$ and $h$, and consider higher order quadrature approximations for both $I_S$ and $I_M$. 

\subsection{Assumptions and higher order result}
We assume that:
\begin{enumerate}
	\item[($HA_{\B}$)] The boundary operator $\B$ is given by \eqref{eq:DirB}, \eqref{eq:NeuB}, or \eqref{eq:RobB}. If $\B(u)$ is the Robin boundary operator, then $\kappa\in C^3(\partial \Omega)$ and $\kappa(x)>0$ for all $x\in \partial \Omega$.
	
	\item[($HA_u$)] $u\in \mathbb{H}_{\mathcal{B}}^{r}(\Omega)$ with $r=\max\{k+1,4\}$.
	\end{enumerate}
	The approximation of the $\B$-fractional Laplacian of $u$ is defined as in \cref{def:discOp} where now $w_h$ is the FE solution of degree $k$ of \eqref{eq:general} with $\theta=1/2$ and the weights $\beta_j$ are given by \eqref{eq:defBetas}-\eqref{eq:funcDer}.

\begin{theorem}\label{thm:HOmain}
Let $s\in(0,1)$, $h>0$, $k\in\mathbb{N}$ and 
\begin{equation}\label{eq:HOrestr}
\Delta t = \eta~h^p \quad \text{for some } \eta>0 \text{ and } p \in \left( 0,(k+1)/{2}\right].
\end{equation}
Assume that \textup{($A_\Omega$), ($HA_u$), ($HA_{\B}$), ($A_{\mathcal{T}_h}$)} hold, and let $\Theta^s_h $ be given by \eqref{eq:discop} where:
\begin{enumerate}[(a)]
\item $w_h$ is the Crank-Nicolson FE solution of degree $k$ of \eqref{eq:general} corresponding to the initial datum $u$;
\item $N_t\geq \dfrac{2-s}{\lambda_{\rm{min}}\varDelta t}\log\left(\dfrac{1}{\varDelta t}\right)$, $\lambda_{\rm{min}}$ being the first nonzero eigenvalue of $(-\Delta_\B)$;
\item the weights $\beta_j$ are defined as
\begin{equation}
\label{eq:defBetas}
\beta_j:= \frac{1}{\varDelta t^s} \times  \left\{
\begin{array}{ll}
\frac{1}{1-s}-F'(1)+F(2)-F(1)  & j=1,\\
F(j+1)-2F(j)+F(j-1) & j=2,\dots,N_t-1,\\
F'(N_t)-F(N_t)+F(N_t-1) & j=N_t,
\end{array}
\right.
\end{equation}
with
\begin{equation}
\label{eq:funcDer}
F(t):= \frac{t^{1-s}}{s(s-1)} \quad \text{and} \quad F'(t):=-\frac{1}{s~t^s},
\end{equation}
and
\[
\beta_\infty:=\int_{t_{N_t}}^\infty \frac{dt}{t^{1+s}}=\frac{1}{s(N_t\varDelta t)^s}.
\]

\end{enumerate}
Then, there exists a constant $C>0$ independent from $h$ such that
\begin{equation}
\label{eq:ofHOMain}
	\|(-\Delta_{\mathcal{B}})^su-\Theta^s_h u\|_{0}\leq C~h^{p(2-s)}.
\end{equation}
\end{theorem}

	\subsection{Proof of \cref{thm:HOmain}}
	\label{sec:prelim}
Similarly to the approach used in \cref{sec:proof}, we use the splitting $(-\Delta_\B)^su=I_S [u]+I_M[u]+I_\infty[u]$ and define different approximations for the singular part, the middle part, and the tail of the integral. Specifically, for the singular part we now define 
\begin{equation}
	\label{eq:singinfpart}
	I_S^D[u](x):=\frac{e^{\varDelta t\Delta_\B}u(x)-u(x)}{\Gamma(-s)\varDelta t}\int_0^{\varDelta t} \frac{dt}{t^{s}}=\frac{1}{\Gamma(-s)}\frac{e^{\varDelta t\Delta_\B}u(x)-u(x)}{(1-s)\varDelta t^{s}}. 
	\end{equation} 

In the following lemmas, we work under the assumptions \textup{$(A_\Omega)$, $(HA_\B)$,} and that the regularity of $u$ is as given explicitly by the constants in each error estimate.	
	
\begin{lemma} \label{lem:sing2}
Let $s\in(0,1)$, $\varDelta t>0$, $I_S$ be as in \eqref{eq:singPart} with $T_1:=\varDelta t$, and $I_S^D[u]$ defined by \eqref{eq:singinfpart}. Then, there exists a constant $c_S=c_S(s,\|u\|_{\mathbb{H}_\B^4})>0$ such that
\[
\|I_S[u]-I_S^D[u]\|_0\leq c_S ~\varDelta t^{2-s}.
\]
\end{lemma}	
\begin{proof}
Using the properties of the exponential, for all $\lambda_m\geq 0$ we have
\[
0\leq \frac{1-e^{-\lambda_m t}}{t}-\frac{1-e^{-\lambda_m \varDelta t}}{\varDelta t}\leq \frac{\lambda_m^2 \varDelta t}{2} \quad \forall t\in[0, \varDelta t].
\]
Hence, 
\begin{equation}
\label{eq:expineq}
\left(e^{-\lambda_m t}-1-\left(\frac{e^{-\lambda_m \varDelta t}-1}{\varDelta t}\right)t\right)^2\leq \frac{\lambda_m^4 \varDelta t^2 t^2}{4} \quad \forall t\in[0, \varDelta t].
\end{equation}
Letting $w=e^{t\Delta_\B}u$, using the series expansion of $w$ and $u$, and from \eqref{eq:expineq}, we obtain
\begin{align*}
\| I_S[u]-I_S^D[u] \|_0 &\leq \frac{1}{|\Gamma (-s)|}\int_0^{\varDelta t}  \left\| w(\cdot,t)-u-\left(\frac{w(\cdot,\varDelta t)-u}{\varDelta t}\right)t \right\|_0 \frac{dt}{t^{1+s}}\\
&\leq \frac{1}{|\Gamma (-s)|}\int_0^{\varDelta t} \left(\sum_{m=1}^\infty \frac{\lambda_m^4\varDelta t^2 t^2}{4}|\hat{u}_m|^2 \right)^{1/2} \frac{dt}{t^{1+s}}\\
&=\frac{\|u\|_{\mathbb{H}_\B^4} \varDelta t}{2|\Gamma (-s)|}\int_0^{\varDelta t}  t\frac{dt}{t^{1+s}}=\frac{\|u\|_{\mathbb{H}_\B^4}}{2|\Gamma(-s)|(1-s)}\varDelta t^{2-s}.
\end{align*}

\end{proof}

For the middle part, we consider the piecewise linear interpolant of $e^{t\Delta_\B}u(x)-u(x)$ on a grid of equally spaced points $t_j=j\varDelta t$, $j=1,\dots,N_t$. Given a function $\phi:[0,+\infty)\to\R$, the linear interpolant of $\phi$ is defined as
\begin{equation*}
\label{eq:pcal}
\mathcal{P}[\phi](t):=\sum_{j=1}^{N_t} \phi(t_j) P_{\varDelta t}(t-t_j), \text{ with }P_{\varDelta t}(t):=\left\{
\begin{array}{ll}
1-\dfrac{|t|}{\varDelta t}& \text{if } |t|\leq \varDelta t, \\
0 & \text{elsewhere}.
\end{array}
\right.
\end{equation*}
Letting $\psi(x,t):=e^{t\Delta_\B}u(x)-u(x)$, we define
	\begin{equation}
	\label{eq:midpart}
	I_M^D[u](x):= \frac{1}{\Gamma(-s)}\int_{\varDelta t}^{N_t\varDelta t} \mathcal{P}[\psi](x,t)\frac{dt}{t^{1+s}}=\frac{1}{\Gamma(-s)}\sum_{j=1}^{ N_t} (e^{t_j\Delta_\B}u(x)-u(x))~ \rho_j,
	\end{equation}
	where $\rho_j:=\int_{\varDelta t}^{N_t\varDelta t} P_{\varDelta t}(t-t_j) \frac{dt}{t^{1+s}}$.
	Note that the integrals $\rho_j$ are positive since the tent functions $P_{\varDelta t}(t-t_j)$ are positive for all $j$. Moreover, following the same procedure used in \cite{HO2014} for the exact computation of quadrature weights we find:
	\begin{equation*}
	\label{eq:defWeights}
	\rho_j=\frac{1}{\varDelta t^s} \times  \left\{
	\begin{array}{ll}
	-F'(1)+F(2)-F(1)  & j=1,\\
	F(j+1)-2F(j)+F(j-1) & j=2,\dots,N_t-1,\\
	F'(N_t)-F(N_t)+F(N_t-1) & j=N_t,
	\end{array}
	\right.
	\end{equation*}
	where $F$ and $F'$ are defined as in \eqref{eq:funcDer}. 

	\begin{remark}Note that the weights $\beta_j$ in \eqref{eq:defBetas} are equal to $\rho_j$ for all $j$ except for $j=1$ where an extra positive term (coming from the discretization of the singular part of the integral in the definition of $(-\Delta_\B)^s$) is added to $\rho_1$. 
	\end{remark}

	\begin{lemma}\label{lem:mid2}
\label{lem:M2}
Let $s\in(0,1)$, $\varDelta t>0$, $I_M$ be as in \eqref{eq:midPart} with $T_1:=\varDelta t$ and $T_2:=N_t\varDelta t$ for some positive integer $N_t$. Let $I_M^D$ be defined by \eqref{eq:midpart}. Then, there exists a constant $c_M=c_M(s,\|u\|_{\mathbb{H}_\B^4})>0$ such that
	\[
	\|I_M[u]-I_M^D[u]\|_0\leq c_M~\varDelta t^{2-s}. 
	\]
\end{lemma}

\begin{proof}
	We start by splitting the integral over $[\varDelta t,N_t \varDelta t]$ as the sum of integrals over subintervals $\mathcal{I}_i= [t_j,t_{j+1}]$ of length $\varDelta t$. Let $w(x,t)=e^{t\Delta_\B}u(x)$ and $\psi(x,t)=w(x,t)-u(x)$. For notational convenience, we denote $g_m(t)=e^{-\lambda_m t}-1$. Then, using \eqref{eq:Wseries} it is straightforward to show that
	\[
	\mathcal{P}[\psi](x,t)=\sum_{j=1}^{N_t} \left(\sum_{m=1}^{\infty} g_m(t)\hat{u}_m \varphi_m(x) \right) P_{\varDelta t}(t-t_j)=\sum_{m=1}^{\infty} \mathcal{P}[g_m](t)\hat{u}_m \varphi_m(x).
	\] 
	and thus,
	\[
	\psi(x,t)-\mathcal{P}[\psi](x,t)=\sum_{m=1}^{\infty} \left(g_m(t)-\mathcal{P}[g_m](t)\right)\hat{u}_m \varphi_m(x).
	\] 
	On each subinterval $\mathcal{I}_i$, $\mathcal{P}[g_m]$ is the Lagrange linear interpolant of $g_m$ and hence there exists a constant $c$ independent of $\varDelta t$ and $t_j$ such that for all $t\in \mathcal{I}_i$,
\[
\|g_m-\mathcal{P}[g_m]\|_{L^\infty([t_j,t_{j+1}])}\leq c\ \varDelta t^2\left\|\frac{\partial^2 g_m}{\partial t^2}\right\|_{L^\infty([t_j,t_{j+1}])} \leq c\ \varDelta t^2\lambda_m^2.
\]
Consequently, for all $t\in \mathcal{I}_i$,
\begin{equation*}
\begin{split}
	\|\psi(t)-\mathcal{P}[\psi](t)\|_0&=\left( \sum_{m=1}^\infty |g_m(t)-\mathcal{P}[g_m](t)|^2|\hat{u}_m|^2\right)^{\frac{1}{2}}\\
	&\leq c\ \varDelta t^2  \left( \sum_{m=1}^\infty \lambda_m^4|\hat{u}_m|^2\right)^{\frac{1}{2}}= c\ \|u\|_{\mathbb{H}_\B^4}\varDelta t^2.
\end{split}
\end{equation*}
	From the definitions of $I_M$ and $I_M^D$, the triangular inequality and Minkowski's integral inequality we hence get
	\begin{align}
	 \|I_M[u]-I_M^D[u]\|_0&\leq \frac{1}{|\Gamma(-s)|} \sum_{j=1}^{N_t} \int_{t_j}^{t_{j+1}} \|\psi(t)-\mathcal{P}[\psi](t)\|_0 \frac{dt}{t^{1+s}}\\
	&\leq \frac{c\ \|u\|_{\mathbb{H}_\B^4} \varDelta t^2}{|\Gamma(-s)|}\int_{\varDelta t}^{N_t\varDelta t} \frac{dt}{t^{1+s}}\leq \frac{c\ \|u\|_{\mathbb{H}_\B^4}}{|\Gamma(-s)|s}~\varDelta t^{2-s}.
	\end{align}
\end{proof}
		
	Finally, we set $I_\infty$ as in \eqref{eq:taildisc} and define
	\begin{equation}\label{eq:OpDisc2}
	\Theta^s u:= I_S^D[u]+I_M^D[u] + I_\infty^D[u].
	\end{equation}
	Then, $\Theta^s u$ can be written as in \eqref{eq:quadrule} with weights $\beta_j$ given by \eqref{eq:defBetas}-\eqref{eq:funcDer}.
		\begin{proposition}
		\label{prop:HOsplit}
		Let assumptions of \cref{thm:HOmain} hold, $(-\Delta_\B)^s$ and $\Theta^s $ be defined by \eqref{eq:start} and \eqref{eq:OpDisc2}. Then, there exists a constant $c_a=c_a(s,\|u\|_{\mathbb{H}_\B^4})>0$ such that
	\[
	\| (-\Delta_\B)^su-\Theta^s u\|_0\leq c_a~\varDelta t^{2-s}.
	\]
\end{proposition}
	
	\begin{proof}
 We proceed as in the proof of \cref{prop:Quad} using now \cref{lem:sing2}, \cref{lem:mid2} and \cref{lem:inf} to get
	\begin{equation*}
		\begin{split}
			\|(-\Delta_\B)^su-\Theta^s u\|_0&\leq \|I_S[u]-I_S^D[u]\|_0+\|I_M[u]-I_M^D[u]\|_0+\|I_\infty [u]-I_\infty^D[u]\|_0\\
			&\leq c_S~\varDelta t^{2-s}+c_M~\varDelta t^{2-s}+c_\infty\frac{e^{-\lambda_{\rm{min}}T_2}}{T_2^s}.
		\end{split}
	\end{equation*}
	Using the definition of $N_t$ and choosing $T_2=N_t\varDelta t$, we get $\frac{e^{-\lambda_{\rm{min}}T_2}}{T_2^s}\leq \varDelta t^{2-s}$. A suitable choice of the constant $c_a$ concludes the proof.
\end{proof}

With \cref{prop:HOsplit} we complete the control of the error coming from the discretization of the integral. 

	\begin{proposition}
		\label{prop:HOM}
		Let assumptions of \cref{thm:HOmain} hold and let $\Theta^s$ be defined by \eqref{eq:OpDisc2}. Then, there is a constant $c_D=c_D(s,\Omega,\|u\|_{k+1},\B,k)>0$ such that
		\[
		\|\Theta^s u - \Theta^s_h u \|_0\leq c_D~( h^{k+1} \varDelta t^{-s} + \varDelta t^{2-s}).
		\]

	\end{proposition}
	\begin{proof}
		We follow the exact same procedure and notation used in the proof of \cref{prop:M}. The bound \eqref{eq:estTri} for the projection error still holds.
The FE error for $t_j>0$ can be now improved exploiting the fact that the second-order Crank-Nicolson discretization scheme is used in time, leading to an error $\mathcal{O}(h^{k+1}+\varDelta t^2)$ (see e.g., Corollary 11.3.1 in \cite{QV2008} and \cite{luskin_etal-aa-1982}). Thus, we find		
\[
\sum_{j=1}^{N_t} \|w(\cdot,t_j)-w_h(\cdot,t_j)\|~ \beta_j\leq c  \left(h^{k+1}+\varDelta t^2\right) \sum_{j=1}^{N_t}  \beta_j.
\]
By definition, $\beta_j>0$ for all $j$. Consequently,
	\[
	\sum_{j=1}^{N_t}\beta_j =\frac{\varDelta t^{-s}}{1-s}+\sum_{j=1}^{N_t}\rho_j= \frac{\varDelta t^{-s}}{1-s}+\frac{F'(N_t)-F'(1)}{\varDelta t^s}\leq c~ \varDelta t^{-s}. 
	\]
		Thus, following the proof of \cref{prop:M}, we conclude that
		 \[
		 	\|\Theta^s u - \Theta^s_h u \|_0\leq c_D~( h^{k+1} + \varDelta t^2) \varDelta t^{-s}.
		 \]
		\end{proof}

We can now prove the higher order discretization result.
\begin{proof}[Proof of \cref{thm:HOmain}]
	Following the same steps done in the proof of \cref{thm:main}, now exploiting the estimates of \cref{prop:HOsplit} and \cref{prop:HOM}, we obtain
\[
	\|(-\Delta_\B)^su-\Theta^s_h u \|_0 \leq c~( h^{k+1} \varDelta t^{-s} + \varDelta t^{2-s}),
\]
	By using the relationship between $h$ and $\varDelta t$ given by~\eqref{eq:HOrestr}, we have $h^{k+1}\varDelta t^{-s}=\eta^{-s}h^{k+1-ps}$ and $\varDelta t^{2-s}=\eta^{2-s}h^{p(2-s)}$. Moreover, $k+1-ps\geq p(2-s)$ for $p\leq \frac{k+1}{2}$. Therefore, the order of convergence is determined by $\varDelta t^{2-s}= h^{p(2-s)}$ and by suitably choosing the constant $C$ in \eqref{eq:ofHOMain} we obtain the desired result.
\end{proof}

\section{Nonlocal operators with non-homogeneous boundary conditions}
\label{sec:inhomBC}
In many practical physical applications, the condition required at the boundary is not necessarily homogeneous. In this section we define and deal with a family of nonlocal operators of fractional type that are compatible with non-homogeneous Dirichlet, Neumann or Robin boundary conditions.

\subsection{The nonlocal operator}
Consider $u:\Omega \to \mathbb{R}$, satisfying $\mathcal{B}(u)=g$ on $\partial \Omega$ for some $g:\partial \Omega\to \R$. Let $v(x,t)$ be the solution of the following heat equation with initial condition $u$ and non-homogeneous boundary conditions:
\begin{equation}
\label{eq:heatg}
\left\{
\begin{array}{lcl}
\partial_t v-\Delta v = 0 & \quad & (x,t)\in \Omega\times (0,\infty),\\
v(x,0)=u(x)& & x\in \Omega,\\
\B(v(\cdot,t))(x)=g(x) & & (x,t) \in \partial \Omega\times [0,\infty).
\end{array}
\right.
\end{equation}
Generalizing \cref{def:fraclap}, we define the nonlocal operator $\LBs$ on the given $u$ as
\begin{equation}\label{eq:fraclapnon}
\LBs [u](x):=\frac{1}{\Gamma(-s)}\int_{0}^\infty \left(v(x,t)-v(x,0)\right) \frac{d t}{t^{1+s}}, \quad \forall x\in \Omega.
\end{equation}

We intentionally avoid using the notations $e^{t\Delta}$ and $(-\Delta)^s$ when dealing with non-homogeneous boundary conditions since $\LBs$ is not a fractional power of the Laplacian coupled to non-homogeneous boundary conditions. However, as we show in Section~\ref{sec:altchar}, $\LBs[u]=(-\Delta_\B)^s[u-z]$, where $z$ is the harmonic extension of $g$ to $\Omega$.

Clearly, the same strategy previously proposed in this work can be followed to obtain numerical approximations of~\eqref{eq:fraclapnon}: the solution of the heat equation~\eqref{eq:heatg} can be approximated via FE accounting for non-homogeneous boundary conditions, while the quadrature formulas for the approximation of the singular integral~\eqref{eq:fraclapnon} remain unchanged. Moreover, the above-mentioned characterization of the nonlocal operator $\LBs$ in terms of $(-\Delta_\B)^s$ can be exploited as a powerful tool in the numerical approach (see Section~\ref{sec:numnonh}).

After the first submission of this manuscript we came across the work by Antil et al.~\cite{antil_etal-arxiv-2017} in which they propose an alternative way to deal with non-homogeneous boundary data. However, in the very recent work by Lischke et al.~\cite{lischke_etal-arxiv-2018}, our definition (given by \eqref{eq:heatg}-\eqref{eq:fraclapnon}) and the one presented in~\cite{antil_etal-arxiv-2017} are shown to be equivalent. Moreover, the authors of~\cite{lischke_etal-arxiv-2018} also present a nice and detailed comparison of the two approaches from the numerical point of view.

\subsection{Characterisation of $\LBs$}\label{sec:altchar}
The solution of \eqref{eq:heatg} can be written as the sum $v(x,t)=w(x,t)+z(x)$, where $w(x,t)=e^{t\Delta_{\B}}[u-z](x)$ is the solution of the heat equation with homogeneous boundary conditions and shifted initial condition, i.e.,
\begin{equation}
\label{eq:heat0}
\left\{
\begin{array}{lcl}
\partial_t w-\Delta w = 0 & \quad & (x,t)\in \Omega\times (0,\infty),\\
w(x,0)=u(x)-z(x)& & x\in \Omega,\\
\B(w(\cdot,t))(x)=0 & & (x,t) \in \partial \Omega\times [0,\infty),
\end{array}
\right.
\end{equation}
and $z$ is the steady-state of \eqref{eq:heatg}, i.e., the solution of the harmonic extension problem
\begin{equation}
\label{eq:ellip}
\left\{
\begin{array}{lcl}
\Delta z = 0 & \quad & x\in \Omega\\
\mathcal{B}(z)(x)=g & & x \in \partial \Omega.
\end{array}
\right.
\end{equation}
The following characterization result provides an alternative definition of $\LBs$.	
\begin{lemma} \label{lem:shift}
	Assume that $u\in H^{2s}(\Omega)$ such that $\B(u)=g$ and let the solution of~\eqref{eq:ellip} be $z\in H^{2s}(\Omega)$. Let also $\mathcal{L}^s_g$ and $(-\Delta_\B)^s$ be defined by \eqref{eq:heatg}-\eqref{eq:fraclapnon} and \eqref{eq:start} respectively. Then, 
	\[
		\LBs [u]=(-\Delta_{\B})^s [u-z] \qquad \textup{in } \Omega.
	\]
\end{lemma}	
\begin{proof}
	Clearly, $u-z \in H^{2s}(\Omega)$ and $\B(u-z)=0$. Let $w(x,t):= e^{t\Delta_{\B}}[u-z](x)$ be the solution of \eqref{eq:heat0}. Then, by using \eqref{eq:ellip} and \eqref{eq:heat0}, for every $x\in\Omega$ the function $v(x,t):=w(x,t)+z(x)$ is such that
	\[
	v_t-\Delta v=w_t-\Delta w+\Delta z=0,
	\] 
	and
	\[
	v(x,0)=w(x,0)+z(x)=u(x)-z(x)+z(x)=u(x).
	\]
	Moreover, it is immediate to see that $\B(v(\cdot,t))=\B(w(\cdot,t))+\B(z)=g$ for every $t\geq 0$ and
	therefore $v(x,t)$ is a solution of \eqref{eq:heatg}. Then,
	\begin{equation*}
		\begin{split}
			(-\Delta_{\B})^s [u-z](x)&=\frac{1}{\Gamma(-s)}\int_{0}^\infty \left(e^{t\Delta_{\B}}[u-z](x)-[u-z](x) \right)\frac{d t}{t^{1+s}}\\
			&=\frac{1}{\Gamma(-s)}\int_{0}^\infty \left(w(x,t)-w(x,0)\right) \frac{d t}{t^{1+s}}\\
			&=\frac{1}{\Gamma(-s)}\int_{0}^\infty \big([w(x,t)+z(x)]-[w(x,0)+z(x)]\big) \frac{d t}{t^{1+s}}\\
			&=\frac{1}{\Gamma(-s)}\int_{0}^\infty \left(v(x,t)-v(x,0)\right) \frac{d t}{t^{1+s}}\\
			&=\LBs [u](x),
		\end{split}
	\end{equation*}
	which concludes the proof.
\end{proof}

The condition requiring the solution to the harmonic extension problem~\eqref{eq:ellip} to be $z\in H^{r}(\Omega)$ for some $r\geq0$ relies heavily on the geometry of the domain $\Omega$, the specific type of boundary conditions $\B$ imposed, and the regularity of the boundary data $g$. Providing a general theory able to encompass all these regularity considerations is far from the main goal of the present article. However, for illustration purposes, we focus on a particular example and show how the numerical tools derived in the previous sections can be applied in these settings. 

\subsection{Numerical approximation in the non-homogeneous case}\label{sec:numnonh}

For simplicity, in this section we restrict ourselves to the case of Dirichlet boundary conditions in a two-dimensional convex polygonal domain $\Omega$ (thus satisfying $\textup{($A_\Omega$)}$). 

Consider a function $u\in H^2(\Omega)$ and let $g=u|_{\partial \Omega}$ be the trace of $u$ in $\Omega$. Classical regularity theory ensures that $g\in H^{\frac{3}{2}}(\partial \Omega)$ and hence, the solution $z$ of \eqref{eq:ellip} is such that $z\in H^2(\Omega)$ (see for example Theorem 1.5 and Theorem 1.8 in \cite{GiRa86}). Thus, $u-z \in H^2(\Omega)$ and $(u-z)|_{\partial \Omega}=0$, via \eqref{eq:Hbchar} we obtain that $u-z\in \mathbb{H}_\B^2(\Omega)$.

In practice, given a function $u$, $z$ will not be a known analytically. To overcome this issue, we consider $\tilde{z}_h$ to be the FE approximation of degree $k= 1$ of problem~\eqref{eq:ellip} for a given triangulation. Namely, $\tilde{z}_h=\sum_{i=1}^{N_h}Z_i \phi_i$, according to the notation previously used in this manuscript. Similarly to \eqref{eq:heat0}, let $\tilde{w}$ be the solution of 
\begin{equation} \label{eq:heatshift2}
\left\{
\begin{array}{ll}
\partial_t \tilde{w}-\Delta \tilde{w} = 0 & \text{in }  \Omega \times (0,\infty)\\
\tilde{w}(x,0)=u(x)-\tilde{z}_h(x)& \text{in }  \Omega \\
\B(\tilde{w}(\cdot,t))(x)=0 & \text{on } \partial  \Omega \times [0,\infty),
\end{array}
\right.
\end{equation}
and $\tilde{w}_h$ be the corresponding FE approximation of degree $k=1$.

The following theorem provides a numerical approximation of $\LBs[u]$ and, exploiting the result of Theorem~\ref{thm:main}, gives a bound for the $L^2$-norm of the error. 

\begin{theorem}\label{thm:inhomBC}
	Let $s\in(0,1)$,  $\theta\in[0,1]$, $k=1$ and $h,\varDelta t>0$ such that \eqref{eq:restr} holds. Let $\Omega\subset \mathbb{R}^2$, assume that \textup{($A_\Omega$)} and \textup{($A_{\mathcal{T}_h}$)} hold, and $\B$ is the Dirichlet boundary operator. Let $u\in H^2(\Omega)$ and $g:=u|_{\partial \Omega}$. Let $\tilde{z}_h$ be the FE solution of degree $k$ of  \eqref{eq:ellip} and $\Theta^s_h$ be defined according to \eqref{eq:discop}, with $N_t$, $\beta_j$ and $\beta_\infty$ given by (b) and (c) from \cref{thm:main}. Then,
	\[
		\|\LBs[u]-\Theta^s_h [u-\tilde{z}_h]\|_{0}\leq C ~ h^{p(1-s)}.
	\]
\end{theorem}
\begin{proof}
	By \cref{lem:shift}, $\LBs [u]=(-\Delta_{\B})^s [u-z]$. Hence, the triangular inequality gives 
	\begin{equation*}
		\begin{split}
			\|(-\Delta_{\mathcal{B}})^{s}[u-z]-\Theta^s_h [u-\tilde{z}_h]\|_0&\leq \|(-\Delta_{\mathcal{B}})^{s}[u-z]-\Theta^s_h [u-z]\|_0\\
			&+\|\Theta^s_h [u-z]-\Theta^s_h [u-\tilde{z}_h]\|_0.
		\end{split}
	\end{equation*}
	Clearly, the result of \cref{thm:main} (for $k=1$) can be applied to obtain 
	\[
		\|(-\Delta_{\mathcal{B}})^{s}[u-z]-\Theta^s_h [u-z]\|_0\leq c_1 ~h^{p(1-s)}.
	\]
	On the other hand, 
	\[
	\|\Theta^s_h [u-z]-\Theta^s_h [u-\tilde{z}_h] \|_0\leq \frac{1}{|\Gamma(-s)|}\sum_{j=0}^{N_t} \left\| w_h(\cdot,t_j)-\tilde{w}_h(\cdot,t_j)-w_h(\cdot,0)+\tilde{w}_h(\cdot,0)\right\|_0 \beta_j ,
	\]
	To bound the right-hand side, we use the stability of the FE method (see e.g. \cite{QV2008} for details) with respect to the initial data, ensuring that for all $j\geq0$, 
	\[
		\|w_h(\cdot,t_j)-\tilde{w}_h(\cdot,t_j) \|_0 \leq \bar{c}~ \| w_h(\cdot,0)-\tilde{w}_h(\cdot,0)\|_0,
	\]
	where $\bar{c}=\bar{c}(\theta)$. At this point we use the linearity of the projection operator $P_h^k$ and the fact that, by definition, $P_h^k$ is exact on each element of the triangulation $\mathcal{T}_h$ for polynomial functions of degree $k$ or less. Hence, by construction, $P_h^k(\tilde{z}_h)=\tilde{z}_h$. Moreover, the accuracy of both the interpolant of $z$ and the FE approximation deriving from the solution of \eqref{eq:ellip} are of order $h^{k+1}$, so that 
	\begin{equation*}\label{blablabla}
		\begin{split}
			\|w_h(\cdot,0)-\tilde{w}_h(\cdot,0)\|_0&=\| P_h^k(u-z)-P_h^k(u-\tilde{z}_h) \|_0 =\| P_h^k(z)-\tilde{z}_h\|_0\\
			&\leq \| z-P_h^k(z)\|_0+ \| z-\tilde{z}_h\|_0\leq c_2~h^{k+1}.
		\end{split}
	\end{equation*}	
	Proceeding as in the previous sections, we combine all estimates and get
	\[
			\|\Theta^s_h [u-z]-\Theta^s_h [u-\tilde{z}_h]\|_0\leq c~ \|w_h(\cdot,0)-\tilde{w}_h(\cdot,0)\|_0 \sum_{j=1}^{N_t}\beta_j  \leq \tilde{c}~h^{k+1} \varDelta t^{-s}.
		\]
	The relationship between $\varDelta t$ and $ h$ given by \eqref{eq:restr} concludes the proof.
\end{proof}

\section{Generalisations}
\label{sec:genL}
The results of \cref{thm:main} and \cref{thm:HOmain} can be generalized not only to non-homogeneous boundary conditions but also to a wider class of operators coupled to homogeneous boundary conditions. 

Instead of the Laplacian operator $-\Delta$ one can consider a more general positive second order elliptic operator of the form
\[
L=\sum_{i,j=1}^n \partial_{x_i}(a_{ij}(x) \partial_{x_j})+ \sum_{i=1}^n b_i(x) \partial_{x_i}+c(x).
\]
Like in \cref{def:fraclap}, the fractional powers of the operator $L$ with the boundary condition $\B$ can be expresed as
\[
(L_\B)^su(x)=\frac{1}{\Gamma(-s)}\int_0^\infty \left( e^{-tL_\B}u(x)-u(x) \right)\frac{dt}{t^{1+s}},
\]
where $e^{-tL_\B}$ denotes the heat-semigroup associated to $L$ and $\B$. Analogous results to \cref{thm:main} and \cref{thm:HOmain} can be proved for the operator $(L_\B)^s$ provided that the bilinear form corresponding to the operator $L_\B$ in the associated variational formulation is weakly coercive (see e.g., \cite{RT1983,Salsa2008} for precise conditions on $L$ and $\B$), and that the coefficients $a_{i j}$, $b_j$, $c$, the domain $\Omega$ and the function $u$ ensure a sufficiently smooth $e^{-tL_\B}u(x)$. 

Regarding the boundary conditions, in this work, we have considered the case of Dirichlet, Neumann, and Robin boundary operators, assuming that a single type of condition is applied to the whole boundary $\partial \Omega$. However, more general boundary constraints (e.g., mixed conditions in different parts of the boundary) can be considered, provided that the properties of weak coercivity of the bilinear form and the regularity of $e^{-tL_\B}u(x)$ needed for the proof can be ensured. 

Details on these regularity conditions for both the operator $L$ and more general boundary conditions can be found in reference \cite{lions_etal-book1968b}.

\section{Numerical experiments}
\label{sec:experiments}
In order to illustrate the performance of our numerical methods we perform various numerical experiments and report here our results. 
The implementations of the methods are done in MATLAB (R2016a, The MathWorks Inc., Natick, Massachusetts, US). 

Let $h, \theta, k$ be fixed parameters of the FE method, and let $\varDelta t$, $N_t$ be chosen according to \cref{thm:main} and \cref{thm:HOmain}. We solve iteratively the linear system $P\mathbf{W}^{(j)}=Q\mathbf{W}^{(j-1)}$, where $P:=(M+\theta \varDelta t A)$ and $Q:=(M+(\theta-1)\varDelta t A)$ are the $N_h\times N_h$ square matrices of the FE scheme, we compute the difference vector $(\mathbf{W}^{(j)}-\mathbf{W}^{(0)})$, multiply it by the weights $\beta_j$, and sum all terms to finally obtain 
\begin{equation} \label{eq:sum}
\frac{1}{\Gamma(-s)}\left[\sum_{j=1}^{N_t}(\mathbf{W}^{(j)}-\mathbf{W}^{(0)})~\beta_j+(\mathbf{W_\infty}-\mathbf{W}^{(0)})\beta_\infty\right].
\end{equation}
Here $\mathbf{W}_\infty$ is the constant vector corresponding to the steady-state of the heat equation (which is known explicitly and depends on the boundary conditions), while $\mathbf{W}^{(0)}$ is the vector with coordinates given by $L^2$-projection of the function $u$ on the nodes of the spatial grid $\Omega_h$. 

To test the convergence of the discretizations, we use the fact that when $\varphi_m$ is a nomalized eigenfunction of $-\Delta$ with the boundary condition $\B$, then the $\B$-fractional Laplacian of $\varphi_m$ can be computed analytically according to  \eqref{eq:eigSerie} as $(-\Delta_\B)^s\varphi_m=\lambda_m ^{ s}\varphi_m$, where $\lambda_m$ is the eigenvalue corresponding to $\varphi_m$. 

\subsection{One dimensional case}
  On the interval $[0,L]$, the eigenfunctions and the corresponding eigenvalues of the Laplacian are known to be \cite{grebenkov_etal-siam-2013}:
	\[
	\begin{array}{lll}
	\varphi_m(x)=\sin\left(\frac{m\pi x}{L}\right) & \lambda_m=\left(\frac{m\pi}{L}\right)^2 & \textup{(Dirichlet case)}\\
	\varphi_m(x)=\cos\left(\frac{(m-1)\pi x}{L}\right) & \lambda_m=\left(\frac{(m-1)\pi}{L}\right)^2 & \textup{(Neumann case)}\\
	\varphi_m(x)=\sin\left(\frac{a_m x}{L}\right)+\frac{a_m}{\kappa L}\cos\left(\frac{a_m x}{L}\right) & \lambda_m=\left(\frac{a_m}{L}\right)^2 & \textup{(Robin case)}\\	
	\end{array}
	\]
	for $m\in \N$. The coefficients $a_m$ in the Robin case are the unique solution of
	\[
	\frac{2a_m}{\kappa L}\cos(a_m)+\left(1-\left(\frac{a_m}{\kappa L}\right)^2\right)\sin(a_m)=0 \quad \textup{for} \quad a_m\in[(m-1)\pi,m\pi].
	\]
	
For~\cref{fig:errAndSol} we select three values of $s\in (0,1)$, namely $s=0.25,0.5,0.75$, and compute the error $\|(-\Delta_{\mathcal{B}})^su-\Theta^s_h u\|_0$ as a function of the spatial discretization parameter $h$, for the Dirichlet, Neumann, and Robin boundary cases. The function $u$ is always a normalized eigenvalue of the corresponding operator, i.e., $u=\varphi_m/||\varphi_m||_0$, for some $m\geq 1$. All reported results were obtained with linear FE ($k=1$), an implicit temporal discretization scheme ($\theta=1$), and $\varDelta t=\eta~h^{k+1}$.

In the left column of~\cref{fig:errAndSol} we compare the expected behavior (dashed lines) predicted by our convergence result \cref{thm:main} with the error decay obtained numerically (solid lines) for different boundary conditions and different values of $s$. On the right, for each boundary condition, we show the exact value of $(-\Delta_\B)^su$ for the three values of $s$ considered. 

\begin{figure}
  \centering
  \subfigure[Error decay]{\includegraphics[width=0.45\textwidth]{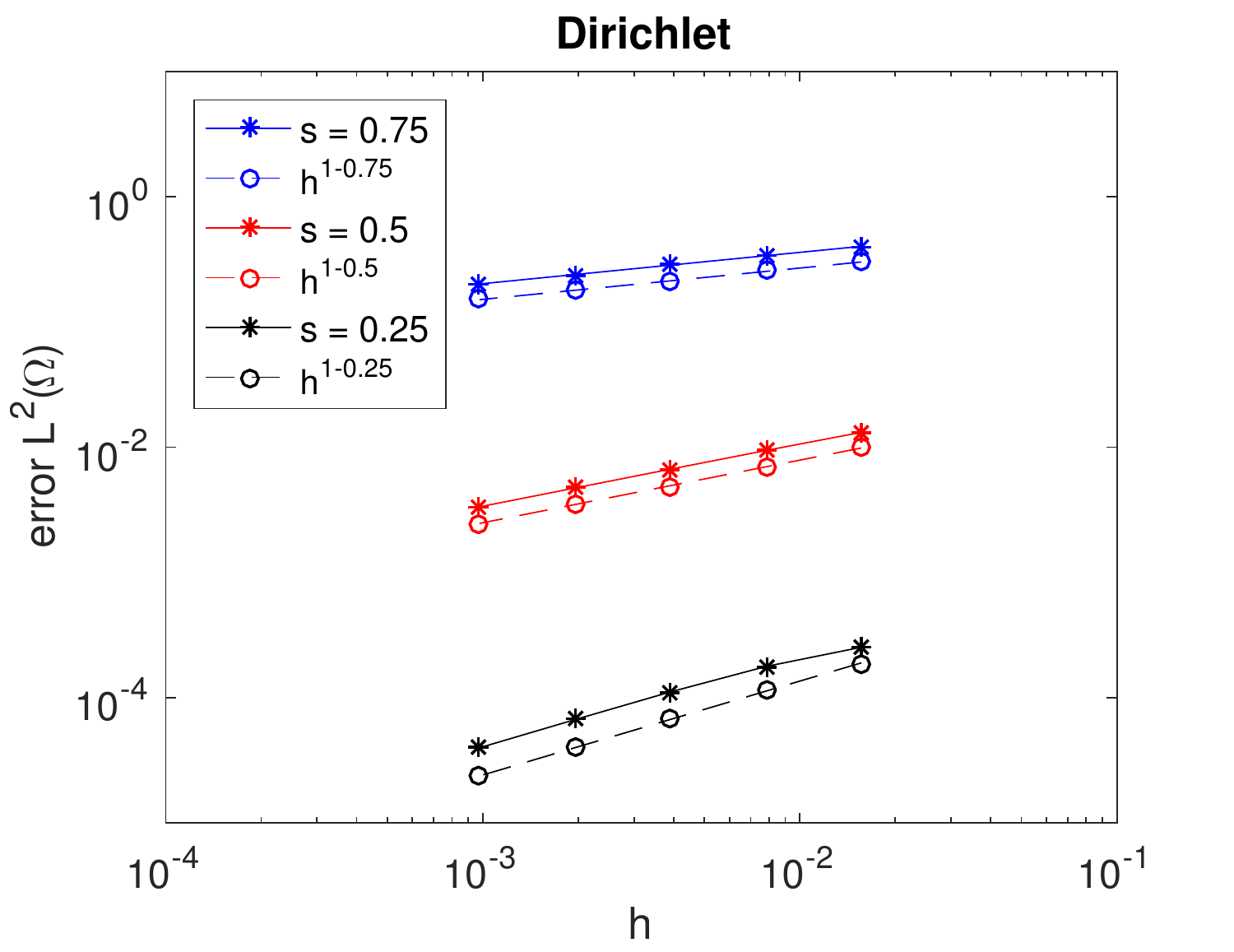}}
  \subfigure[Dirichlet fractional Laplacian, $m=1$]{\includegraphics[width=0.45\textwidth]{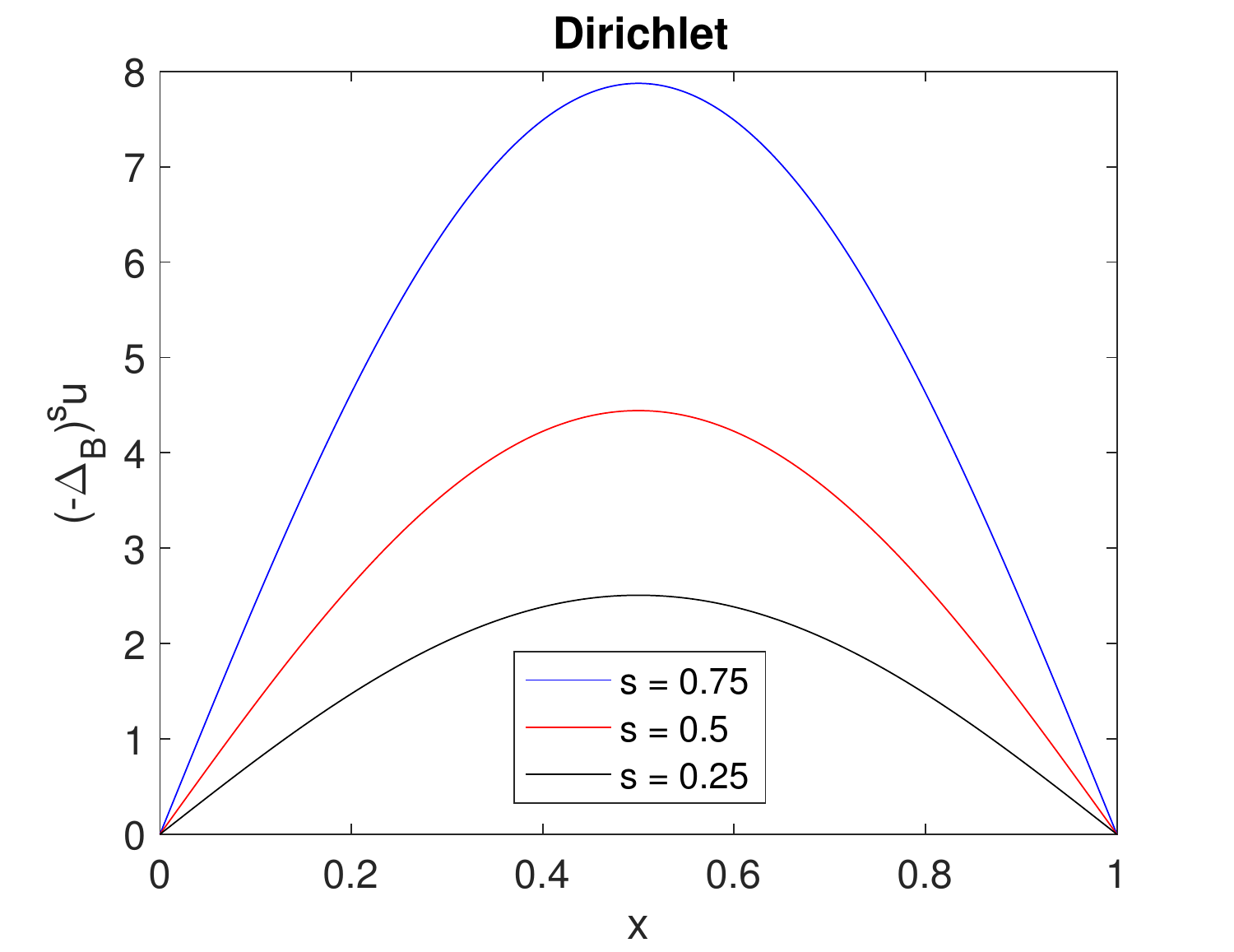}}
  \subfigure[Error decay]{\includegraphics[width=0.45\textwidth]{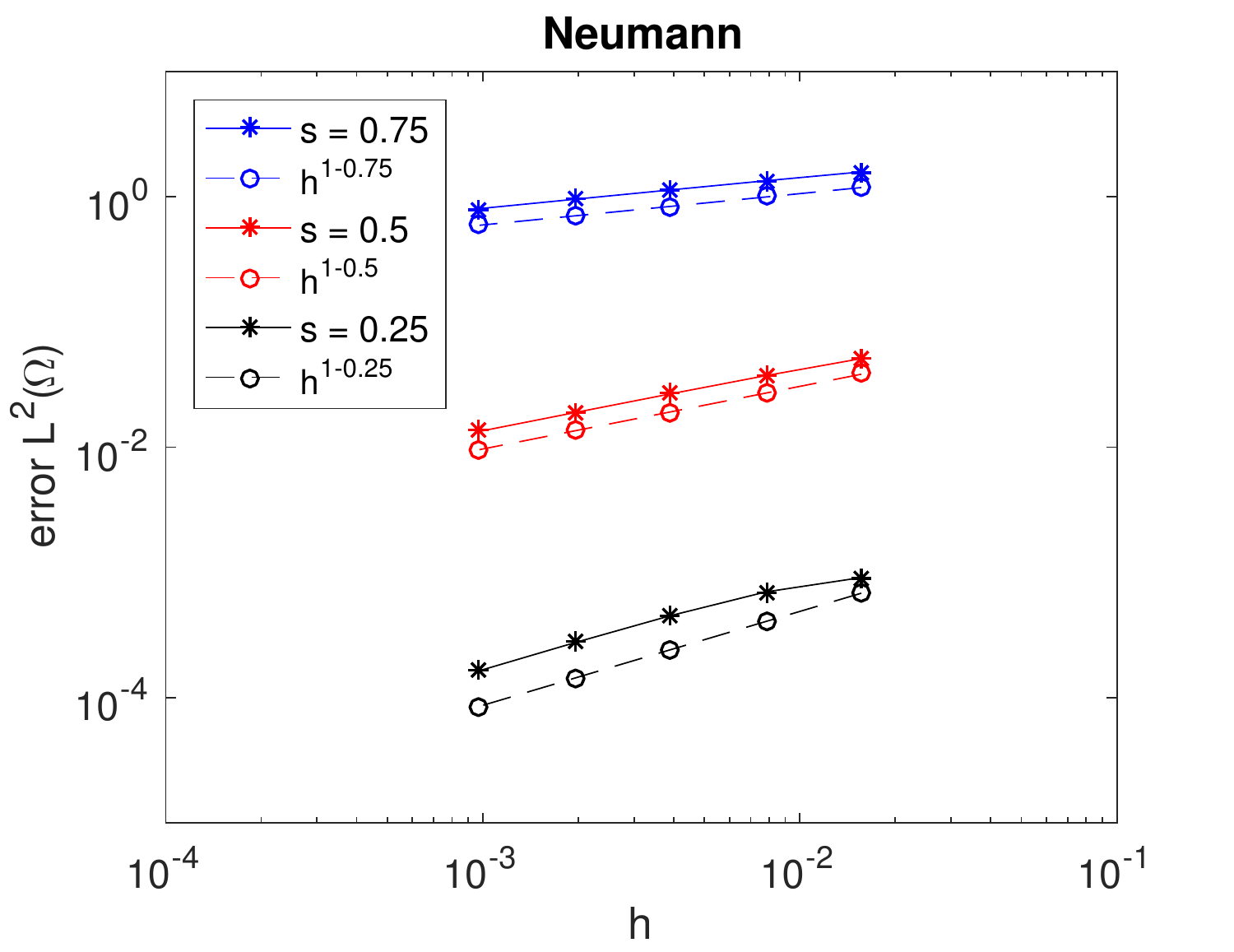}}
  \subfigure[Neumann fractional Laplacian, $m=3$]{\includegraphics[width=0.45\textwidth]{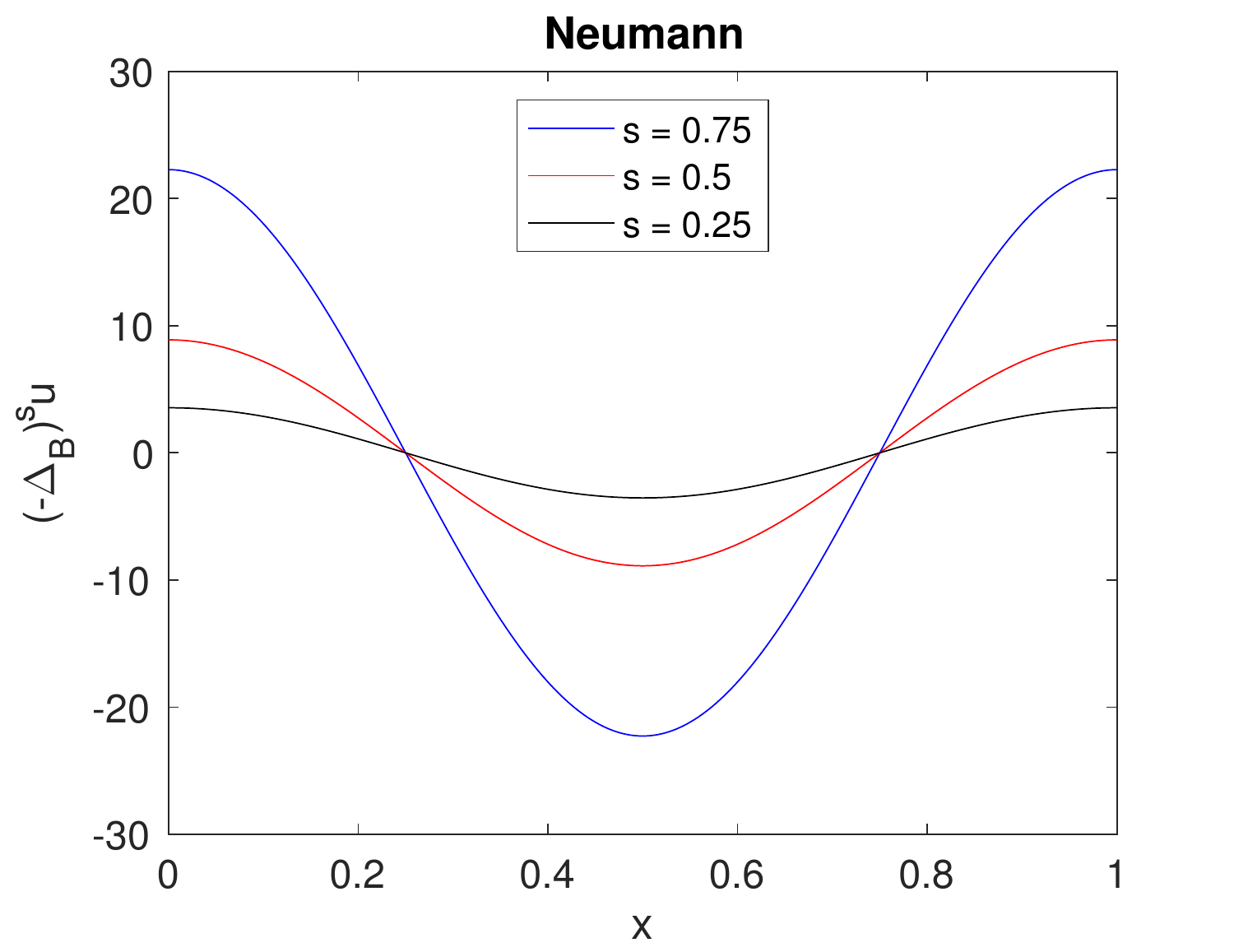}}
  \subfigure[Error decay]{\includegraphics[width=0.45\textwidth]{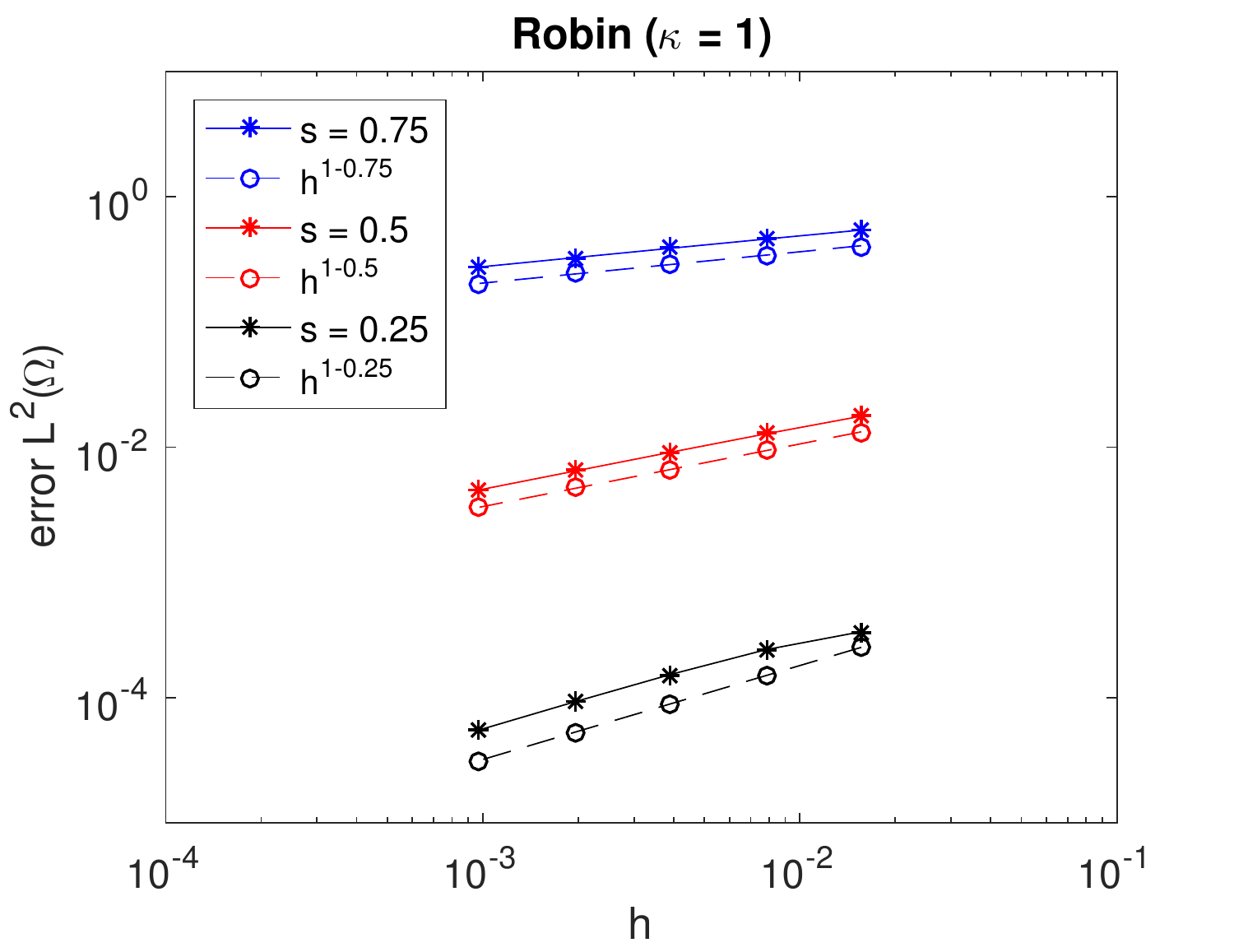}}
  \subfigure[Robin fractional Laplacian, $m=2$, $\kappa=1$]{\includegraphics[width=0.45\textwidth]{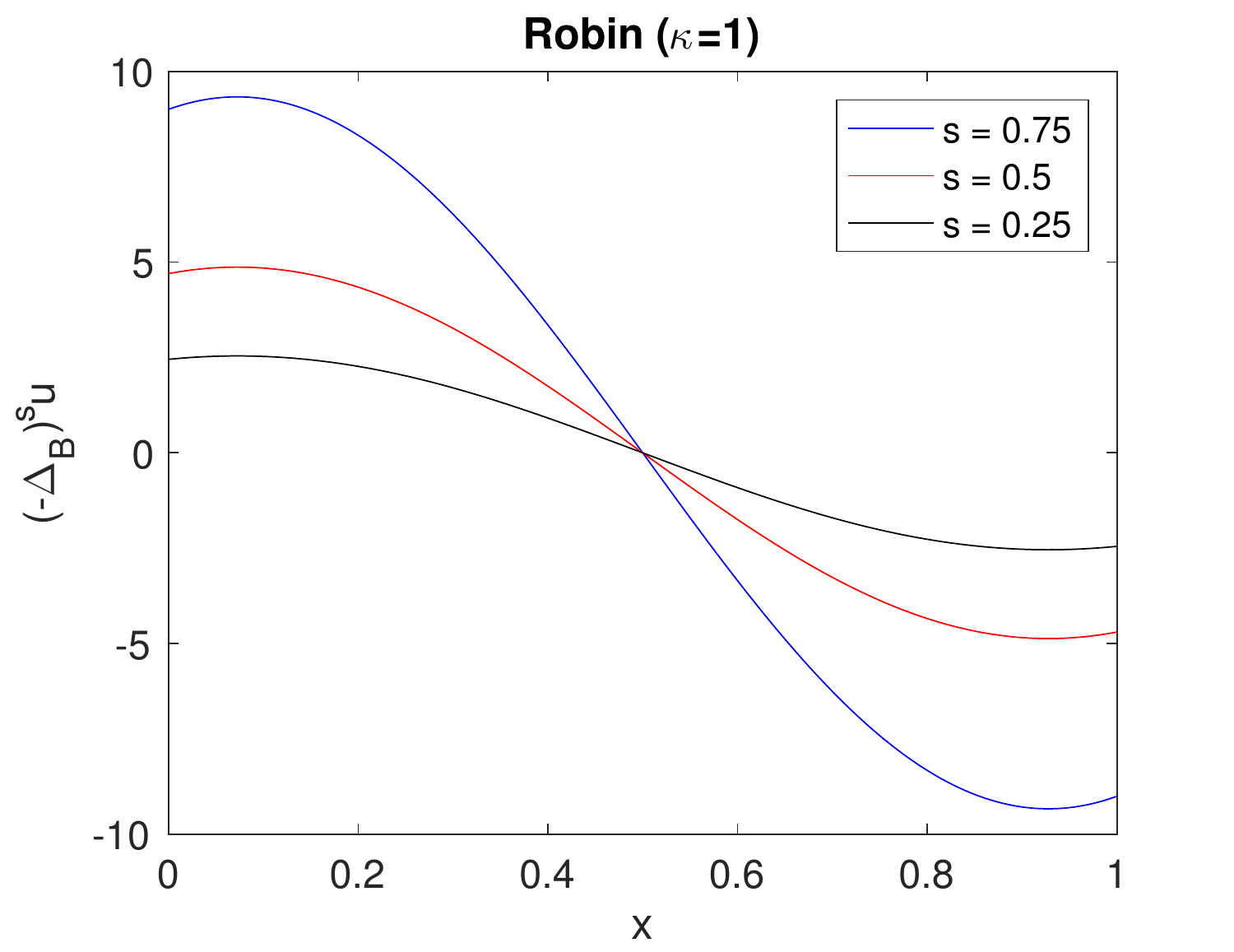}}
  \caption{Error convergence and $(-\Delta_\B)^su(x)$ on $\Omega=(0,1)$ for the Dirichlet (top row), Neumann (mid row), and Robin (bottom row) cases. In all cases $u=\varphi_m/||\varphi_m||_0$ is a normalized eigenfunction of $(-\Delta_\B)$. The value of $m$ is as indicated in (b), (d), and (f), respectively. Parameters of the experiment: $k=1$, $\theta=1$, $p=1$, $\eta=0.001$.} 
  \label{fig:errAndSol}
  \vspace{-0.5cm}
\end{figure}	

  In~\cref{fig:error1Dhighs} we show the improved convergence obtained for the same functions considered in the tests of \cref{fig:errAndSol} when the higher order discretization method described in \cref{sec:HO} is used. We observe that the convergence decay actually outperforms the trend predicted from our theoretical results for all values of $s$ and that the behavior improves as $s$ increases.

\begin{figure}
	\centering
	\subfigure{\includegraphics[width=0.45\textwidth]{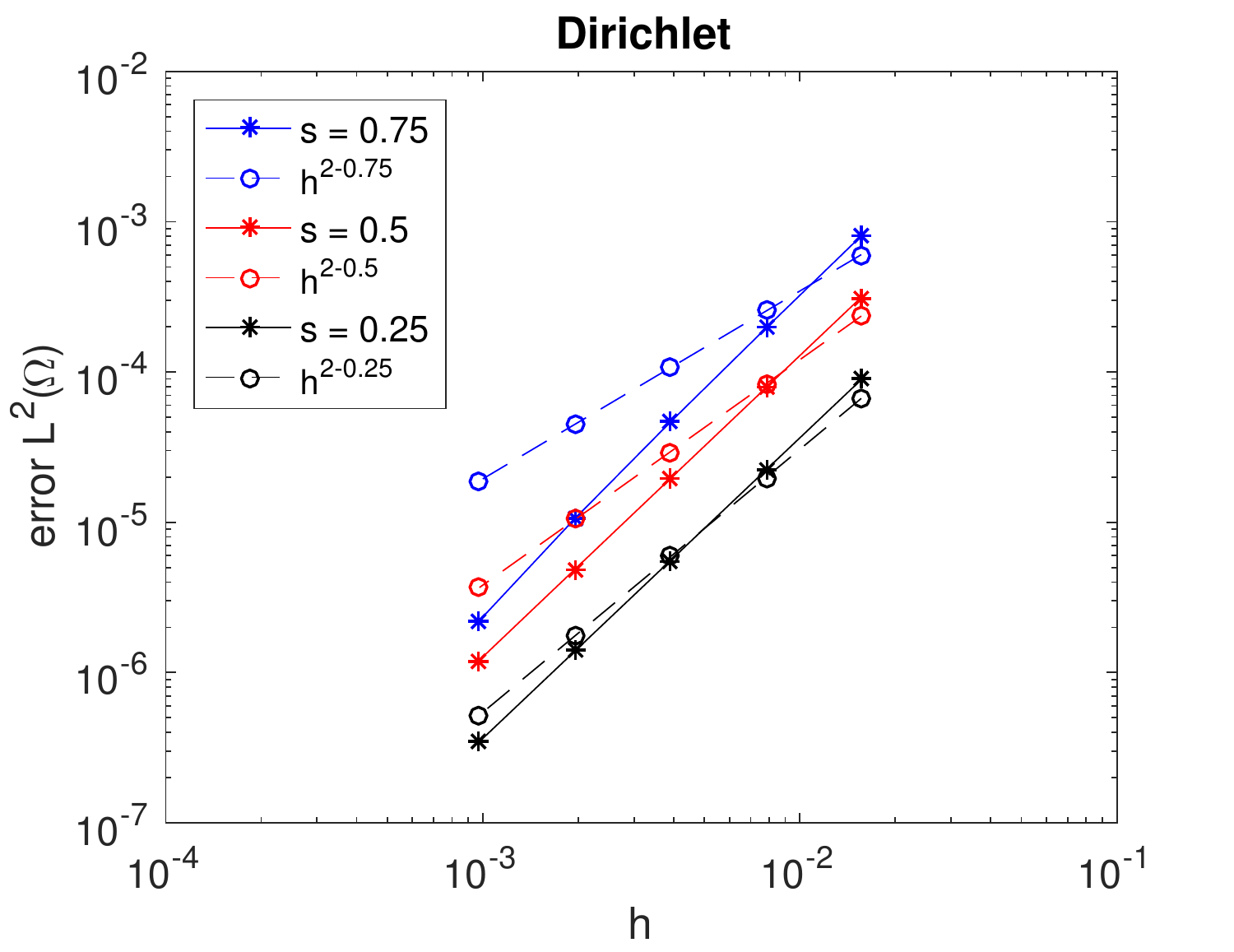}}
	\subfigure{\includegraphics[width=0.45\textwidth]{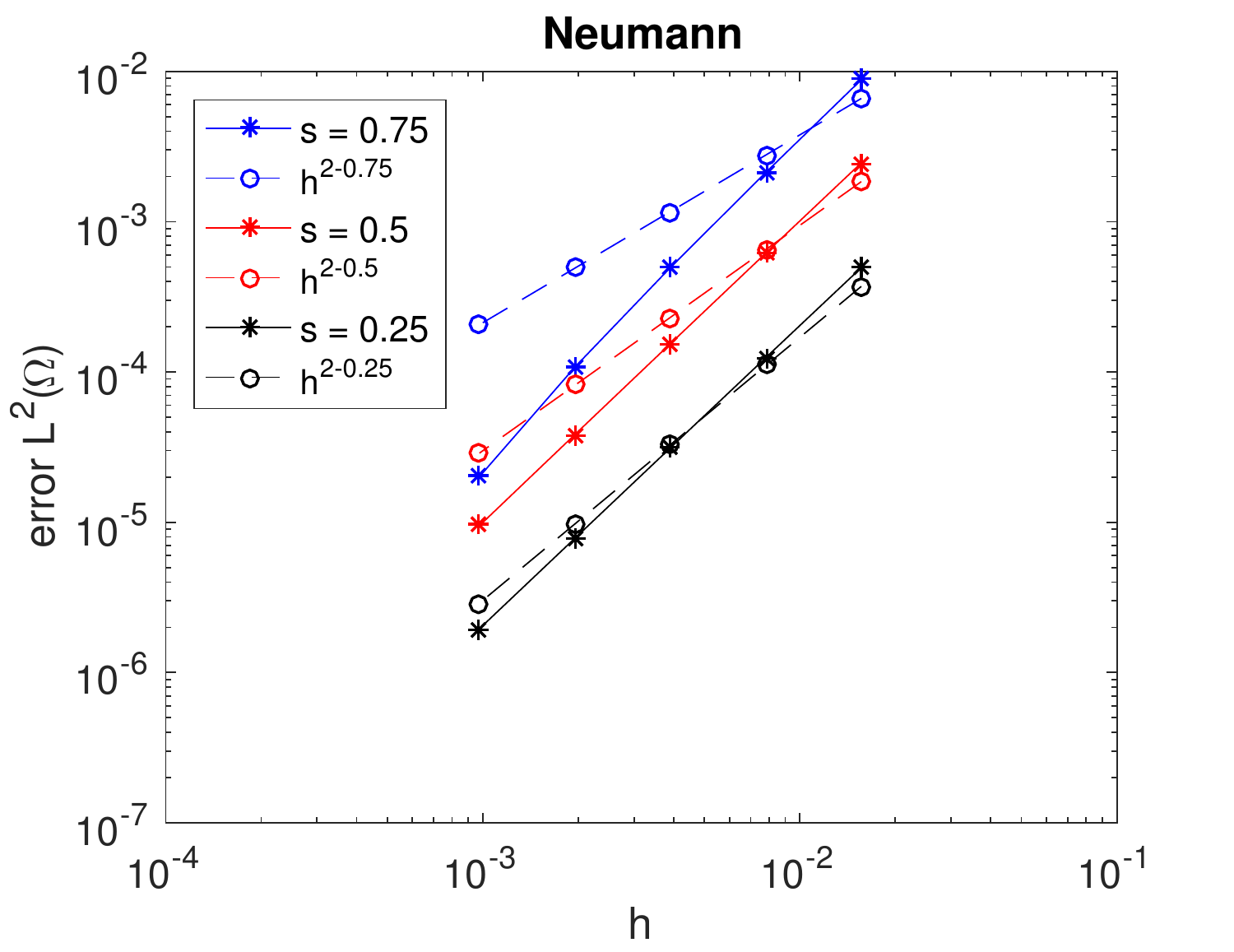}}
	\subfigure{\includegraphics[width=0.45\textwidth]{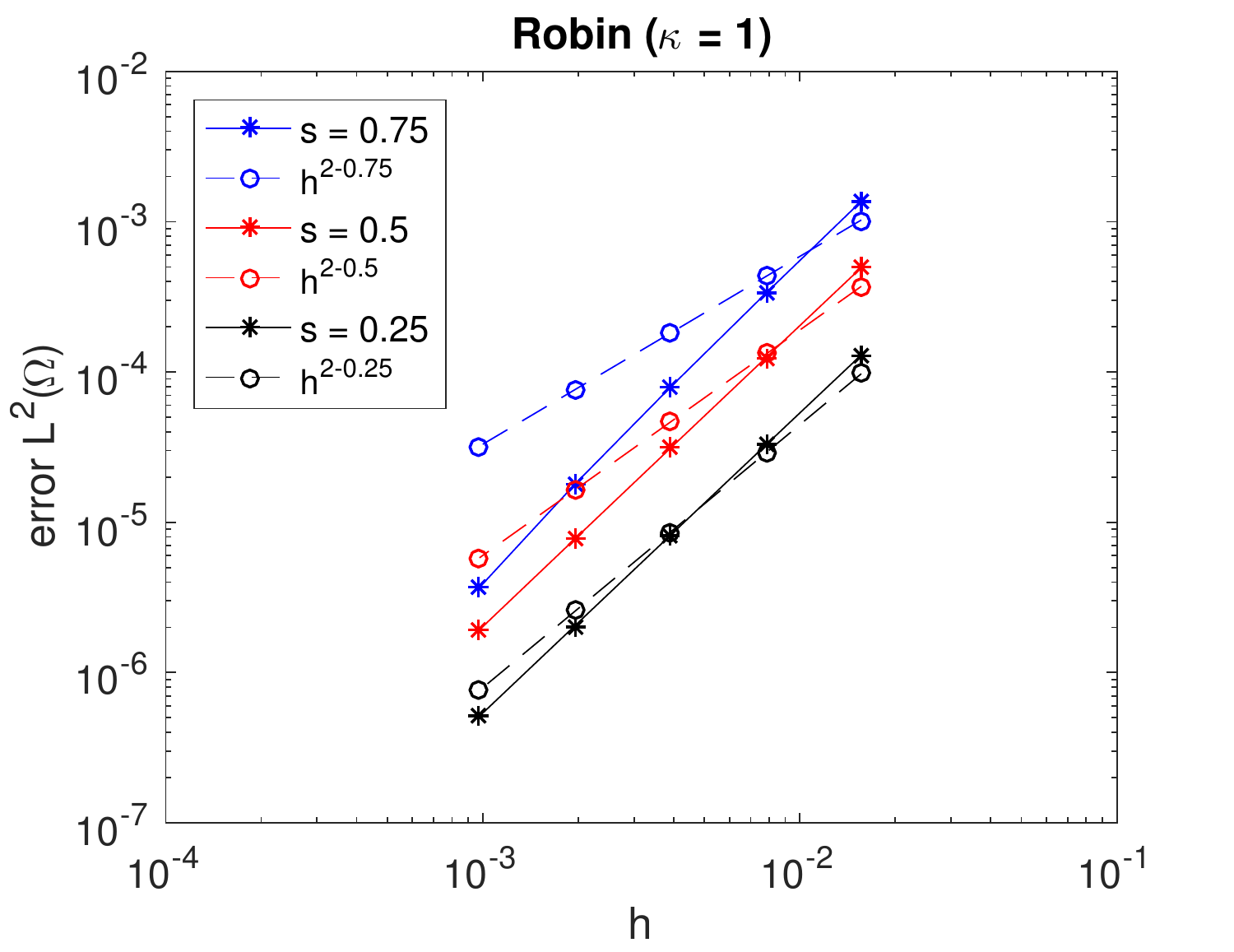}}
	\caption{Error convergence for $\B$-fractional Laplacian of the same eigenfunctions considered in~\cref{fig:errAndSol} on $\Omega=(0,1)$ computed with the higher order method, for values of $s=0.75,0.5,0.25$. Parameters of the experiment: $k=1$, $\theta=1/2$, $p=1$, $\eta=0.001$.}
	\label{fig:error1Dhighs}
\end{figure}

As mentioned various times in this work, the $\B$-fractional Laplacian given by \cref{def:fraclap} is a boundary-dependent nonlocal operator, that is, the operator changes when different boundary conditions $\B$ are considered. This can be seen by applying the operator for different choices of $\B$ to the same function $u$. Letting $\mathbbm{1}_D$ denote the indicator function of the set $D$, in ~\cref{fig:solExpCompare}, we consider a compactly supported function of the form $u(x)=e^{-1/(r^2-x^2)}\mathbbm{1}_{|x|<r}$, with $0<r<1$,
on $\Omega=(-1,1)$, and apply to it the $\B$-fractional Laplacian with homogeneous Dirichlet, Neumann, and Robin boundary conditions for different values of the fractional parameter $s$. The fractional Laplacian is computed via the higher order method.

\begin{figure}
	\centering
	\subfigure{\includegraphics[width=0.45\textwidth]{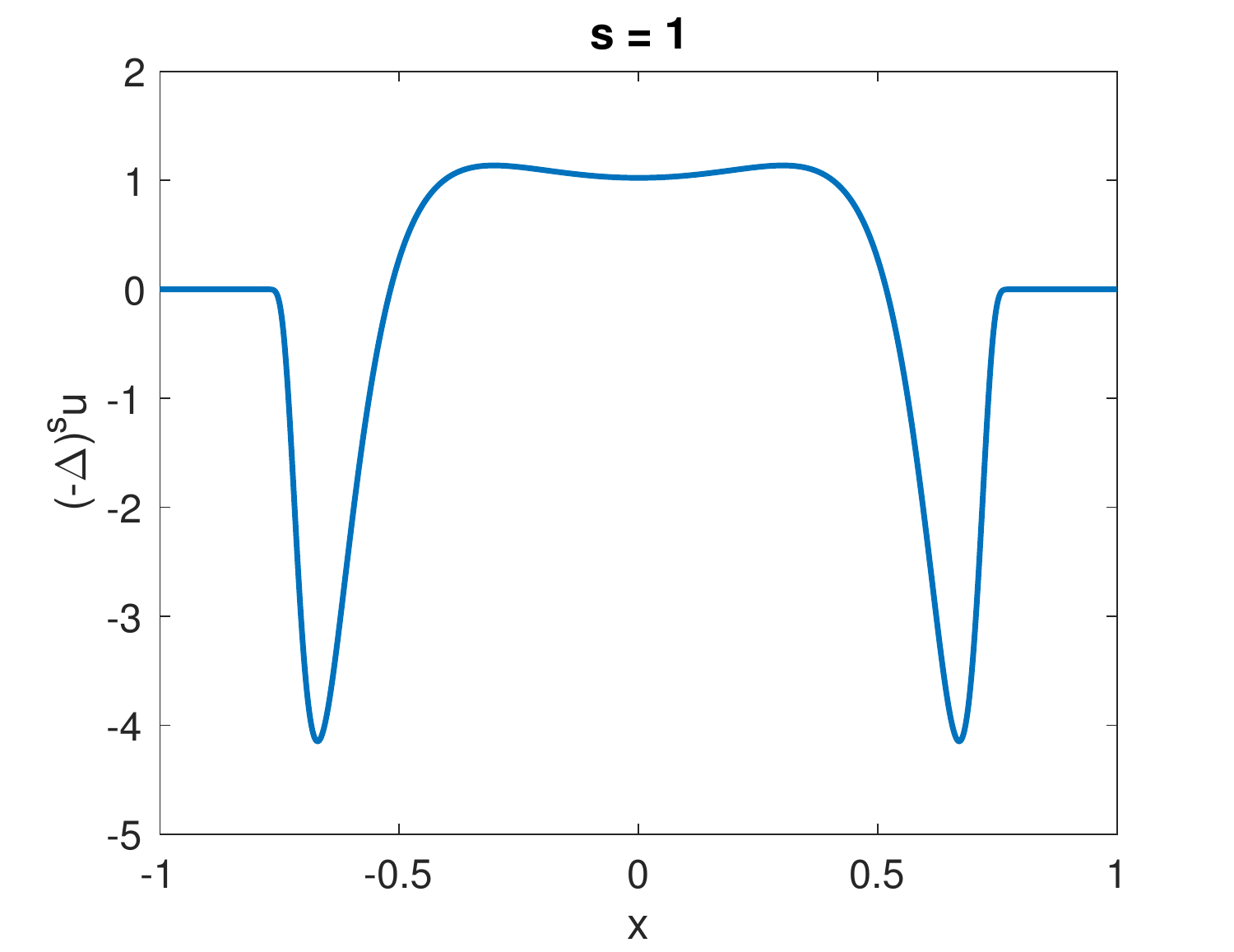}}
	\subfigure{\includegraphics[width=0.45\textwidth]{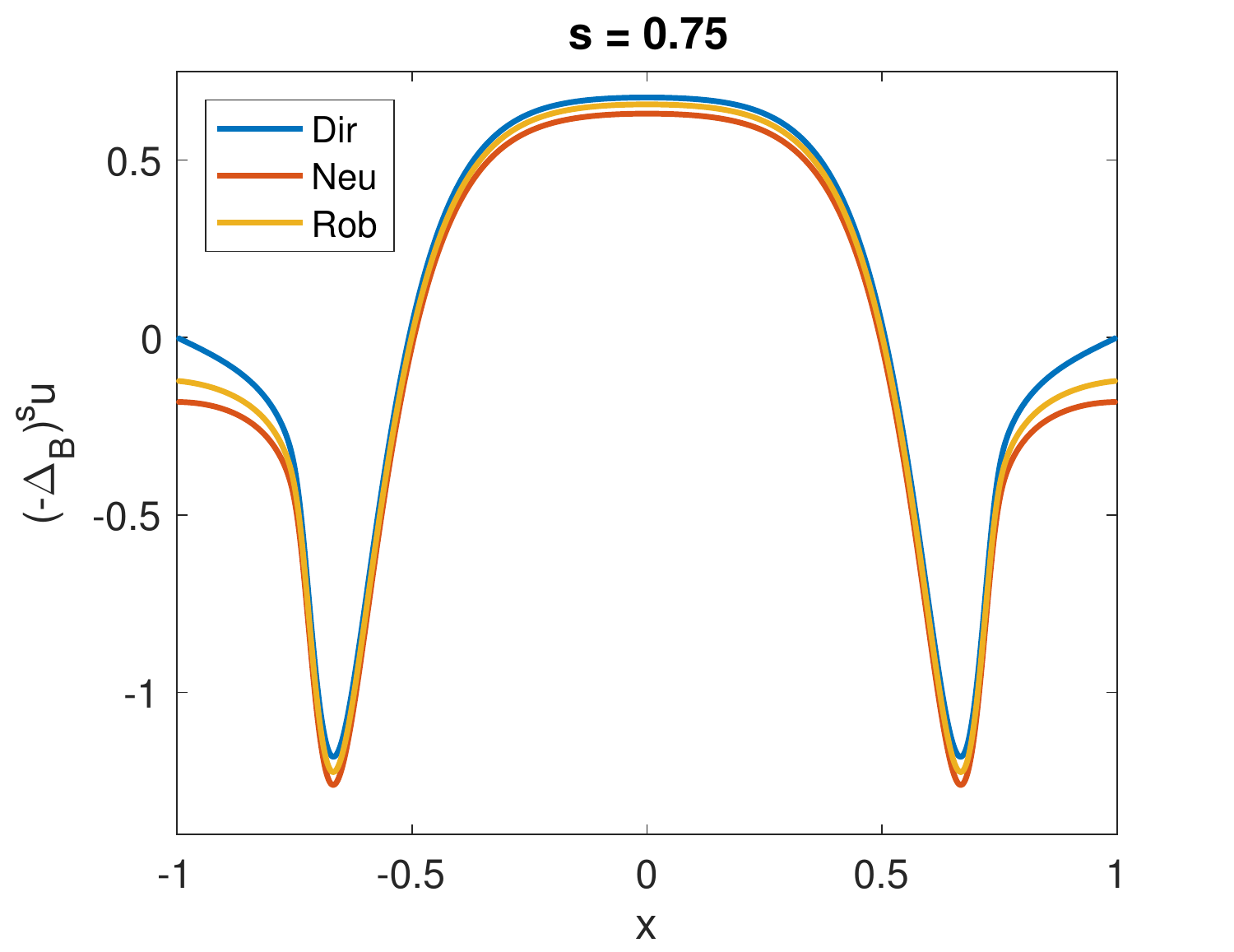}}
	\subfigure{\includegraphics[width=0.45\textwidth]{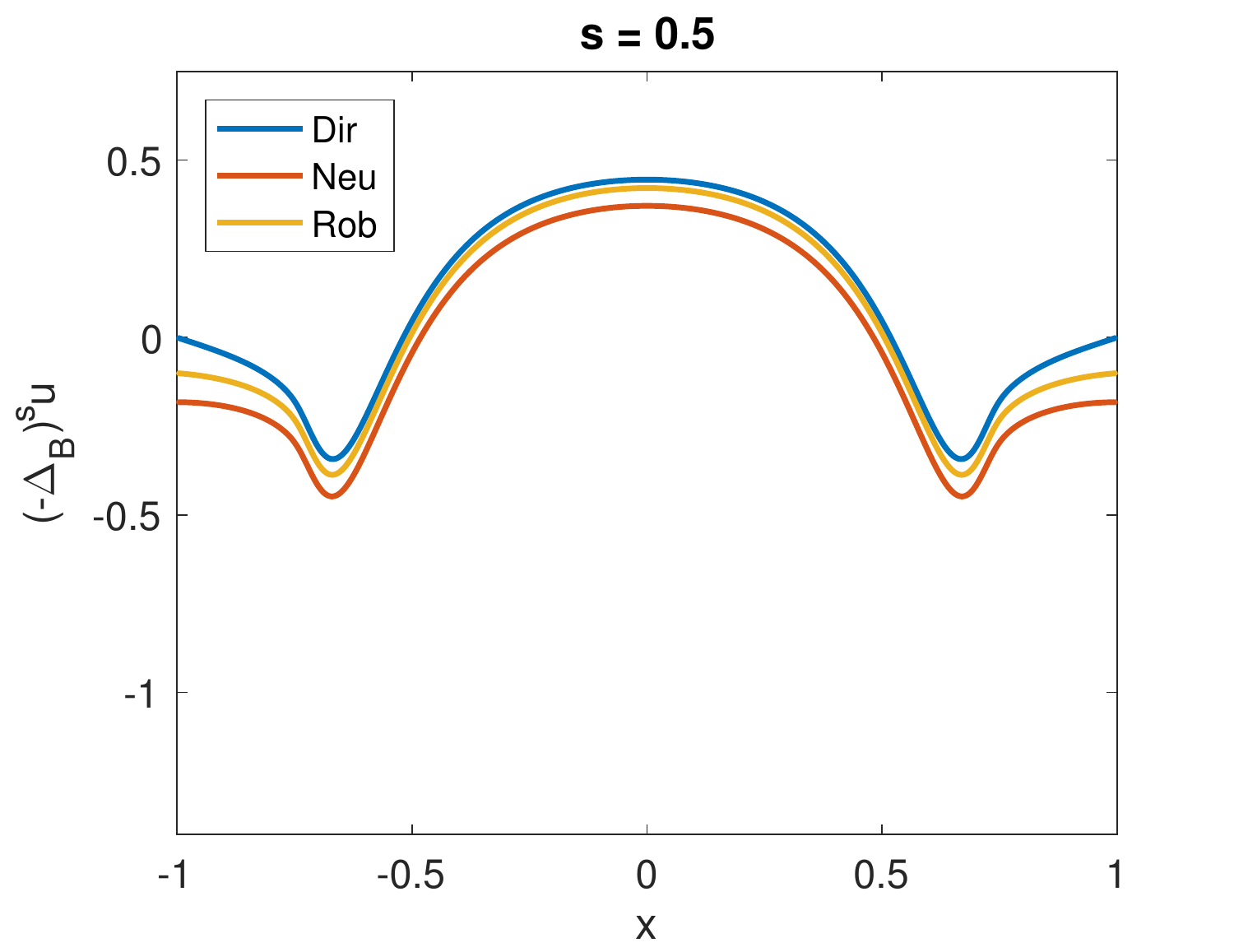}}
	\subfigure{\includegraphics[width=0.45\textwidth]{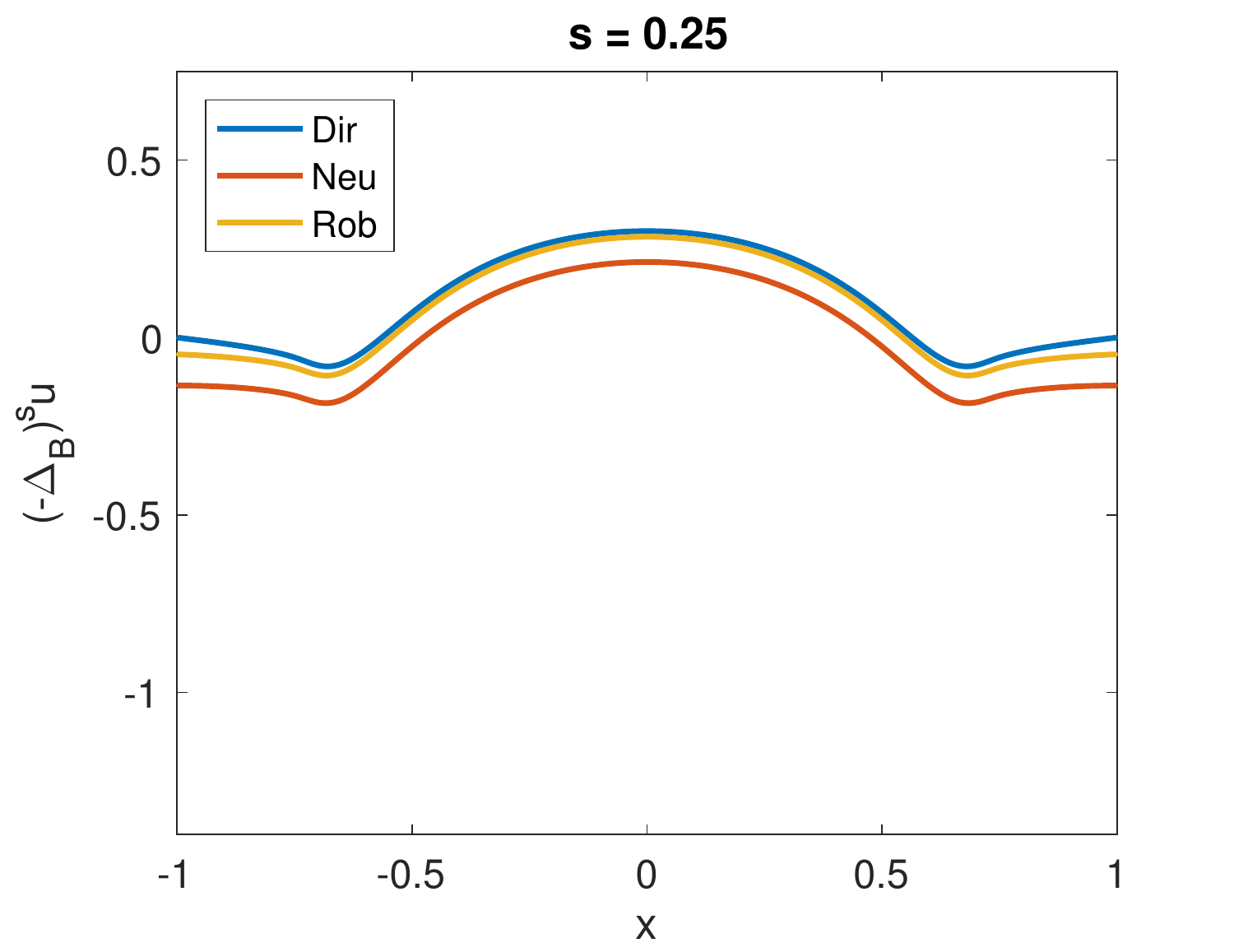}}
	\caption{Standard Laplacian ($s=1$) and $\B$-fractional Laplacian with different boundary conditions applied to the same function for $s=0.75, 0.5, 0.25$. $u(x)=exp(-1/(r^2-x^2))$ on $[-r,r]$, $r=0.8$, and zero elsewhere, $\Omega=(-1,1)$. Parameters of the experiment: $k=1$, $\theta=1/2$, $p=1$, $\eta=0.001$}.
	\label{fig:solExpCompare}
\end{figure}	

For comparison purposes, in the top left of \cref{fig:solExpCompare} we plot the standard Laplacian of the considered $u$, which is boundary independent. As we can see, the standard Laplacian of a compactly supported function is still a compactly supported function, whereas this does no longer hold in the purely fractional cases. The nonlocal character of the operator in fact affects the result and produces functions that are nonzero also on $\Omega\setminus [-r,r]$. Moreover, differences between the three nonlocal operators considered are clearly visualized for all choices of $s<1$.

\subsection{Two dimensional case}
The spectrum of the classical Laplacian is known analytically also in the two-dimensional case for some simple regular geometries, such as rectangle, disk, or ellipse (see \cite{grebenkov_etal-siam-2013} for further details). To prove convergence of our method in the two-dimensional case we exploit these analytical results and use them to compute $\|(-\Delta_{\mathcal{B}})^su-\Theta^s_h u\|_0$ as a function of the spatial discretization parameter $h$, for the usual three different boundary conditions $\B$, on the unit square $\Omega=(0,1)^2$. In this case, eigenfunctions are simply products of the one-dimensional eigenfunctions corresponding to the particular boundary conditions considered, namely
\[
\varphi_{ml}(x,y)=\varphi_m(x)\varphi_l(y),
\]
and the corresponding eigenvalues are given by the sum of the one-dimensional counterparts, that is, $\lambda_{ml}=\lambda_m+\lambda_l$.

The convergence results and the index of the particular eigenfunction $\varphi_{ml}$ used for the test are given in ~\cref{fig:error2Dsquare}.

\begin{figure}
	\centering
	\subfigure[$m=l=1$]{\includegraphics[width=0.45\textwidth]{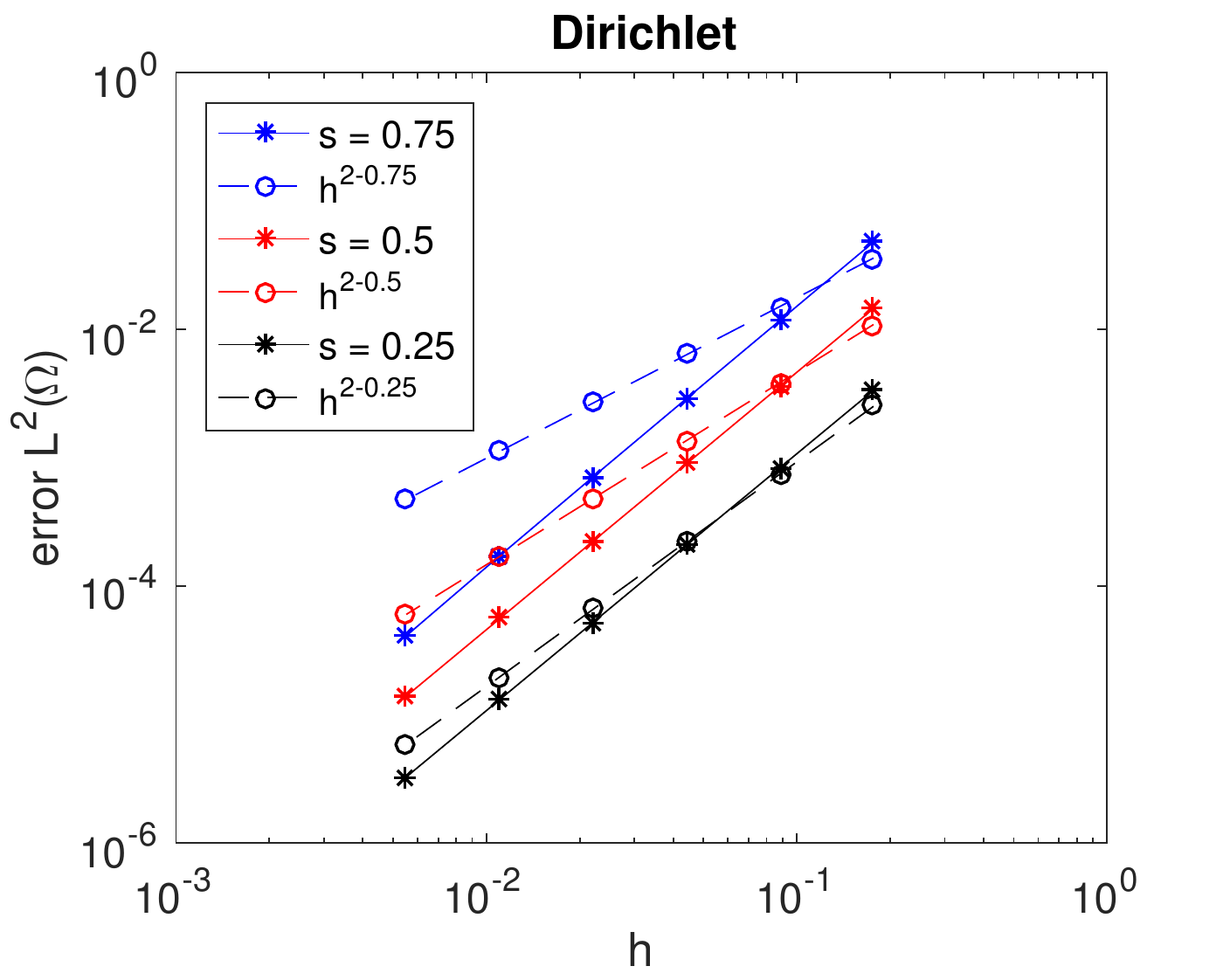}}
	\subfigure[$m=l=3$]{\includegraphics[width=0.45\textwidth]{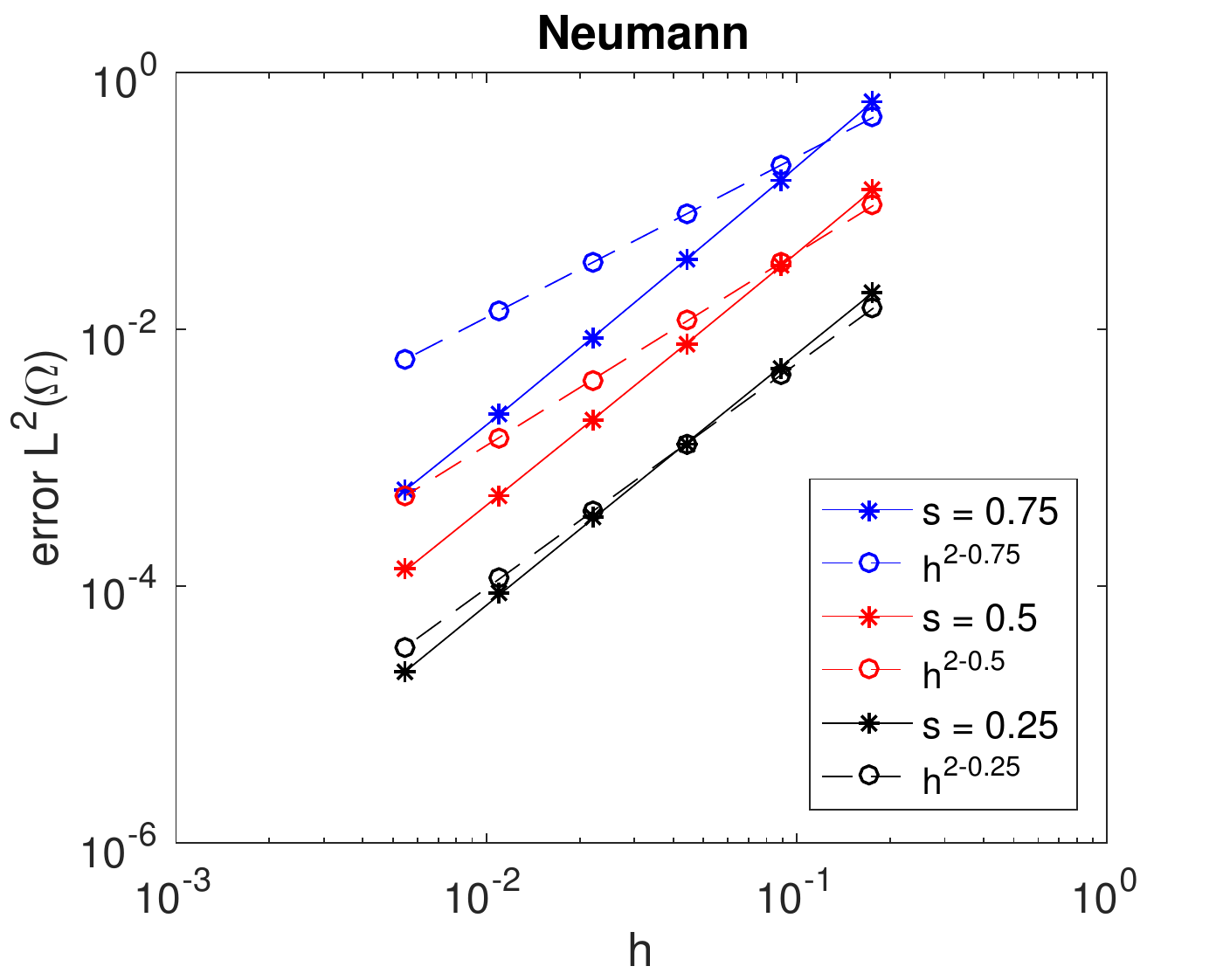}}
	\subfigure[$m=2,l=1$]{\includegraphics[width=0.45\textwidth]{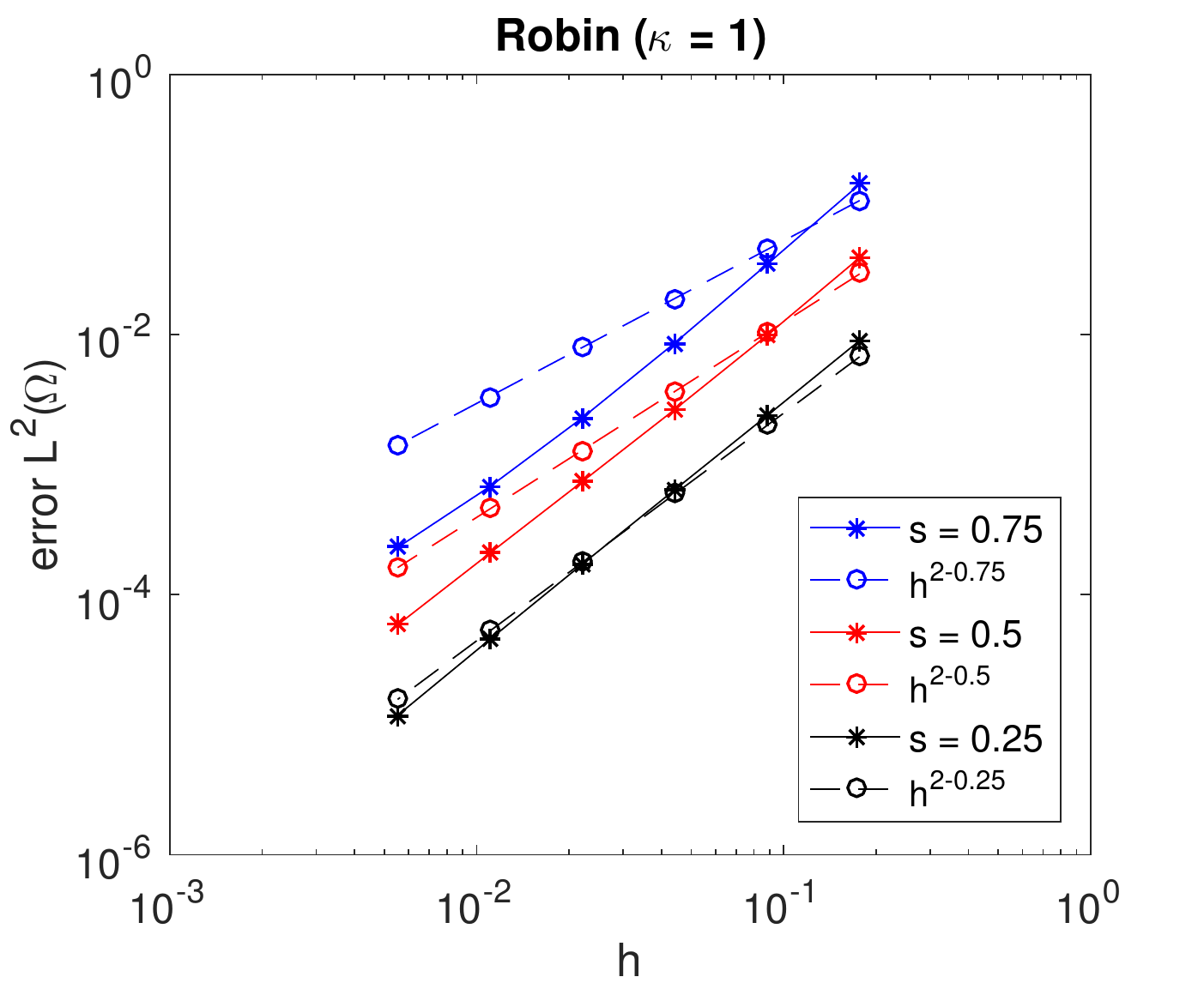}}
	\caption{Error convergence for $\B$-fractional Laplacian on the unit square. The indices of the considered eigenfunction $\varphi_{ml}$ are as indicated in (a), (b), (c). Parameters of the experiment: $k=1$, $\theta=1/2$, $p=1$, $\eta=0.001$}.
	\label{fig:error2Dsquare}
\end{figure}	

In \cref{fig:inhomogCase} we provide an example of the use of our method to compute the action of the fractional operator $\mathcal{L}^s_{g}$ on a function $u$ satisfying non-homogeneous Dirichlet boundary conditions on the unit square $\Omega=(0,1)^2$. Specifically, we set $g(x,y)=0.1\sin (2\pi x)\cos(\pi y)$ for $(x,y)\in \partial \Omega$. Wanting to study convergence of the method we let $\varphi:=u-z$ be a normalized eigenfunction of the homogeneous Dirichlet Laplacian on the unit square so that $(-\Delta_{\B})^s \varphi$ is known explicitly.

Not knowing the analytic expresssion of $z$, we compute its FE approximation $z_h$ and hence obtain the FE approximation $u_h=\varphi+z_h$ (satisfying non-homogeneous boundary conditions) to be used as initial datum for the approximate computation of $\LBs$. 
The solution is computed with the higher order method and the error once again decays faster than expected for all three values of $s$ considered. 

\begin{figure}
	\centering
	\subfigure{\includegraphics[width=0.45\textwidth]{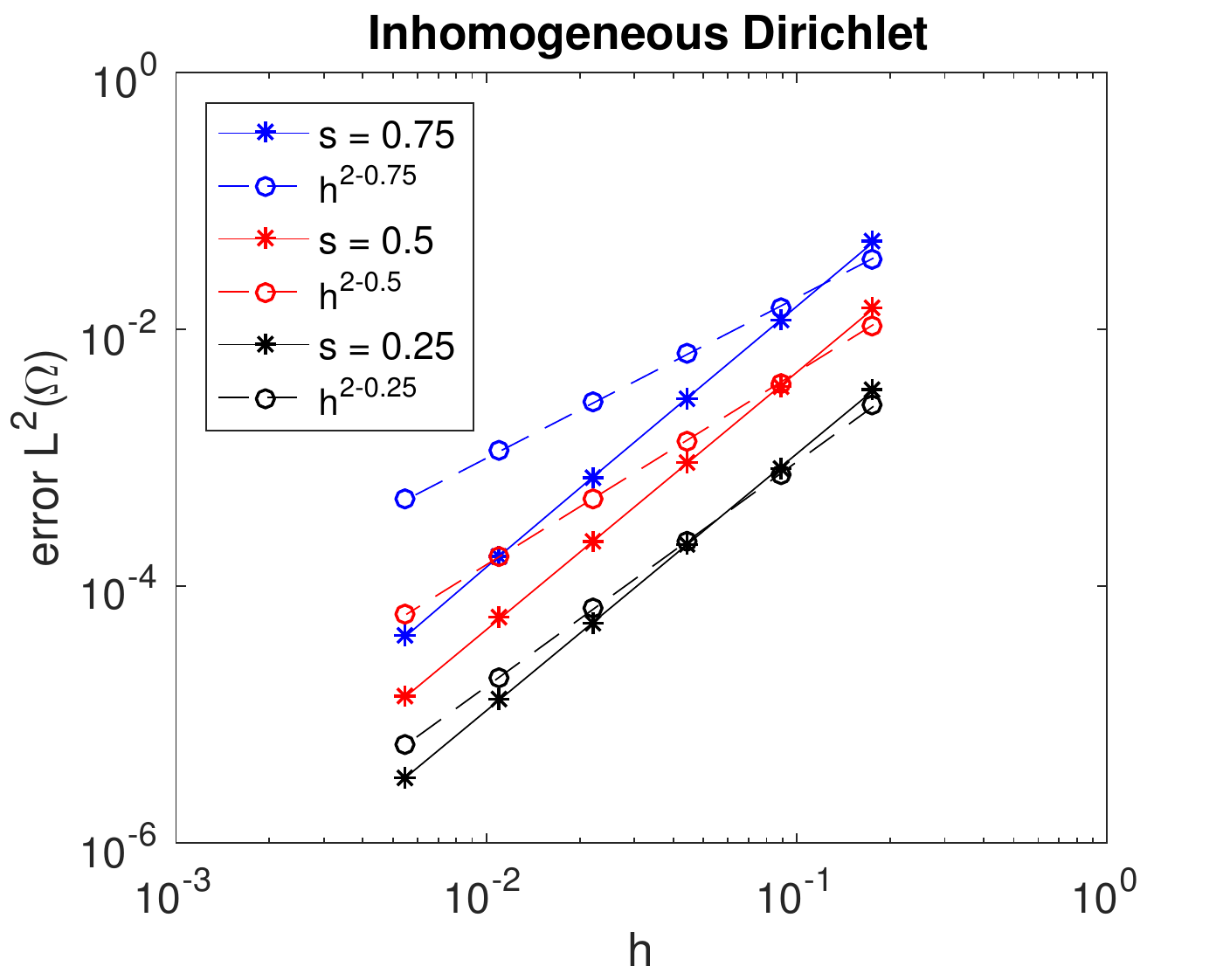}}
	\subfigure{\includegraphics[width=0.45\textwidth]{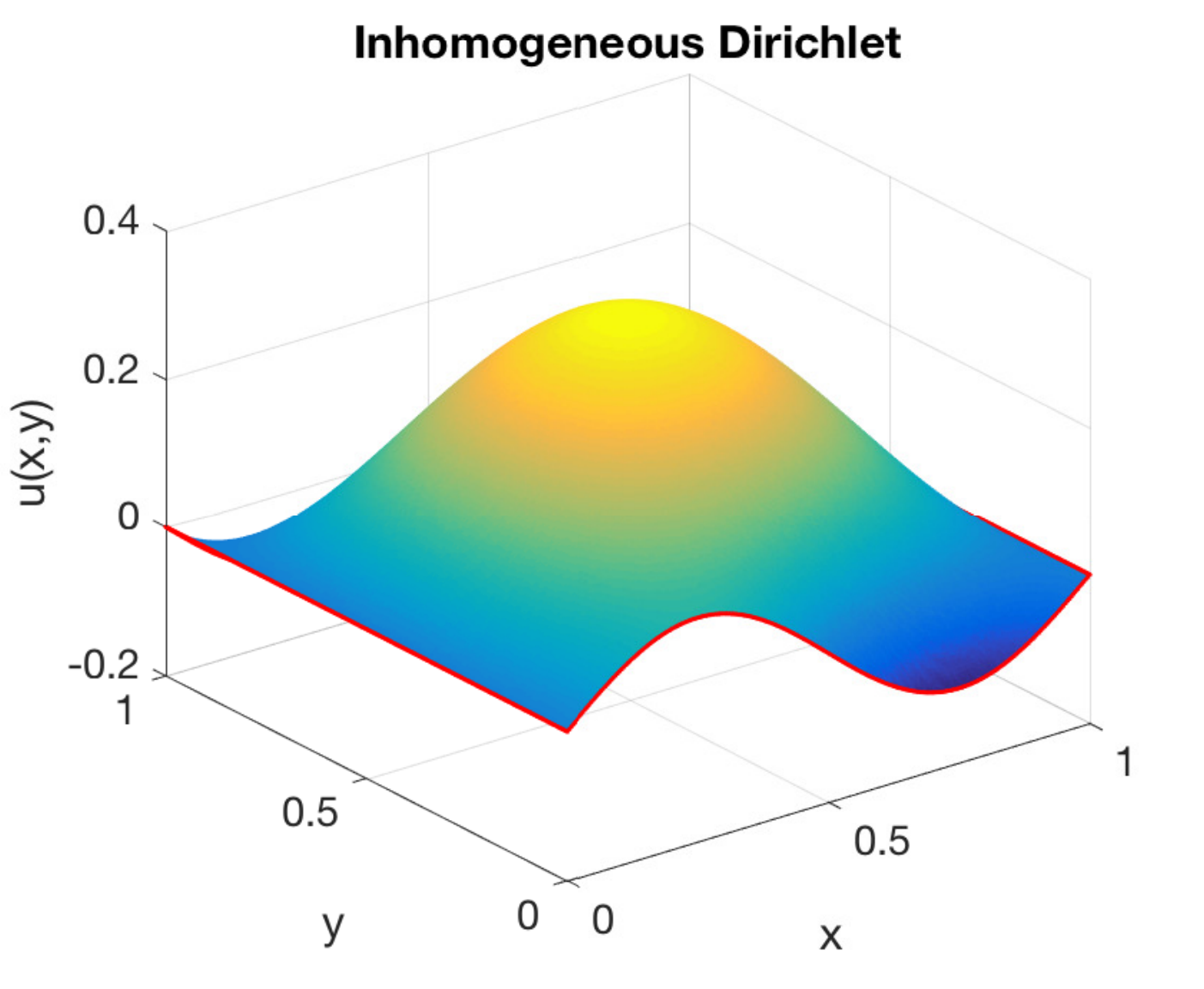}}
	\subfigure{\includegraphics[width=0.45\textwidth]{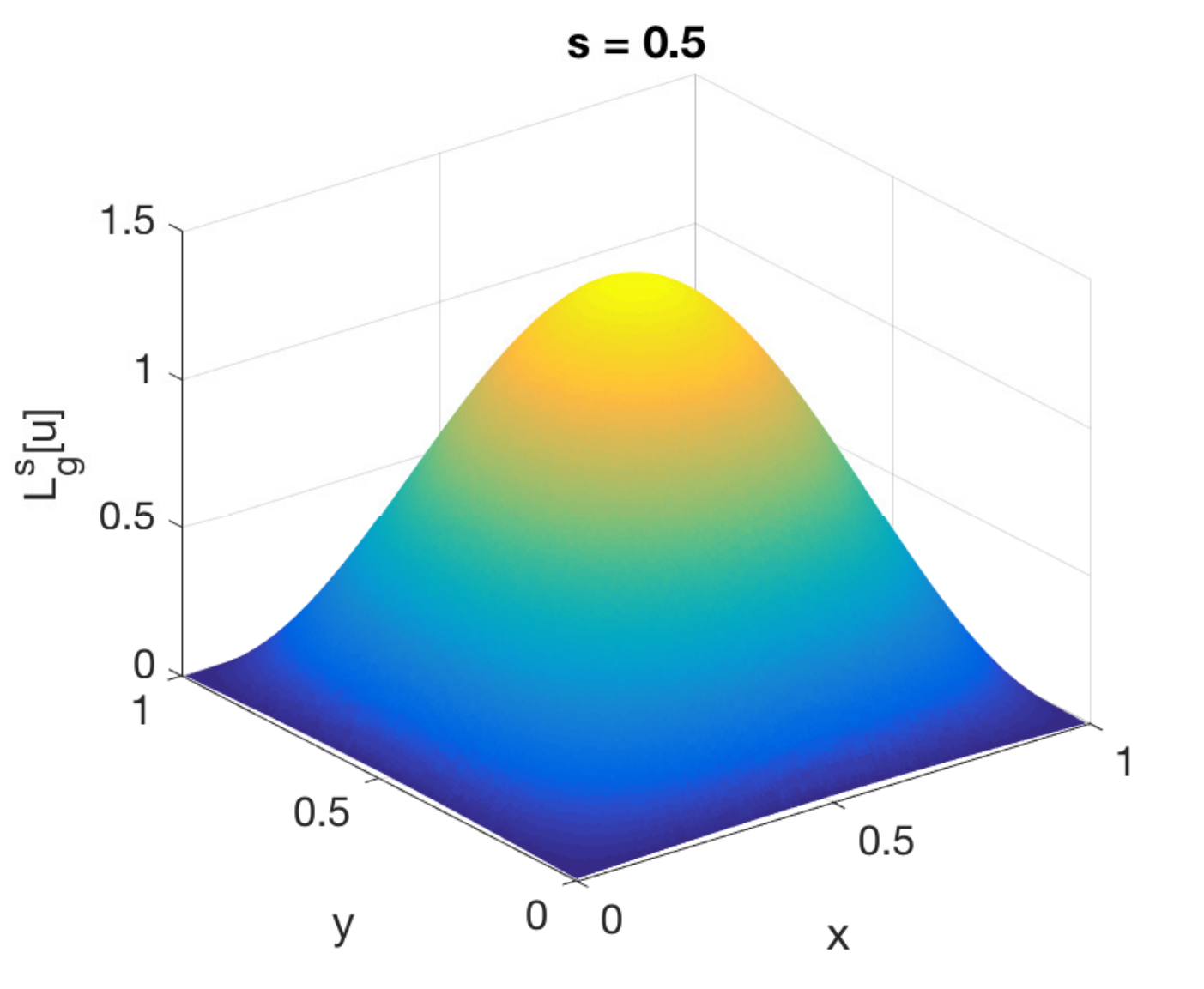}}
		\caption{Error convergence, function $u$, and nonlocal operator $\LBs $ with non-homogeneous Dirichlet boundary conditions applied to $u$ in the case $s=0.5$. The domain is the unit square. The boundary function is $g(x,y)=0.1\sin(2\pi x) \cos(\pi y)$. Boundary nodes are highlighted in red to help with visualization. Parameters of the experiment: $k=1$, $\theta=1/2$, $p=1$, $\eta=0.001$.}
	\label{fig:inhomogCase}
\end{figure}

\subsubsection{Polygonal domain}
As mentioned in the introduction, one could compute the $\B$-fractional Laplacian on a bounded domain $\Omega$ via equation \eqref{eq:spectraldef}. However, the analytic expression of the required eigenpairs is not known for general domains and using \eqref{eq:spectraldef} would imply that the solution of a very large number of eigenvalue problems must be computed.
	
One of the main advantages of our approach and of the use of the numerical method proposed in this work is that with the same strategy we can handle boundary conditions that are not necessarily homogeneous Dirichlet conditions and at the same time we do not need to know the eigendecomposition of the Laplacian on the given domain to compute $(-\Delta_\B)^su$. Exploiting \eqref{eq:semigr} and due to the versatility of the FE method in handling general geometries, we are able to compute $(-\Delta_\B)^su$ on domains such as the one considered in \cref{fig:Comp2DIrregDom}, for which the eigenpairs are analytically unknown.
Although in our theoretical results, the choice of $N_t$ depends on the first nonzero eigenvalue of the Laplacian, clearly the same results hold if a lower bound is used for the considered eigenvalue. These bounds are readily available and we refer the reader to Section 4.2 of \cite{grebenkov_etal-siam-2013} and references therein for a thorough discussion on these bounds for general domains and homogeneous Dirichlet, Neumann, and Robin boundary conditions. An alternative approach not using bounds for the first non-trivial eigenvalue simply consists in continuing to add terms to the quadrature sum in \eqref{eq:sum} until the difference $(\mathbf{W}^{(j)}-\mathbf{W}^{(0)})$ reaches a steady value, up to a given tolerance. The latter approach is the one we adopted for the computations presented in \cref{fig:Comp2DIrregDom}.

\begin{figure}
	\centering
	\includegraphics[width=\textwidth]{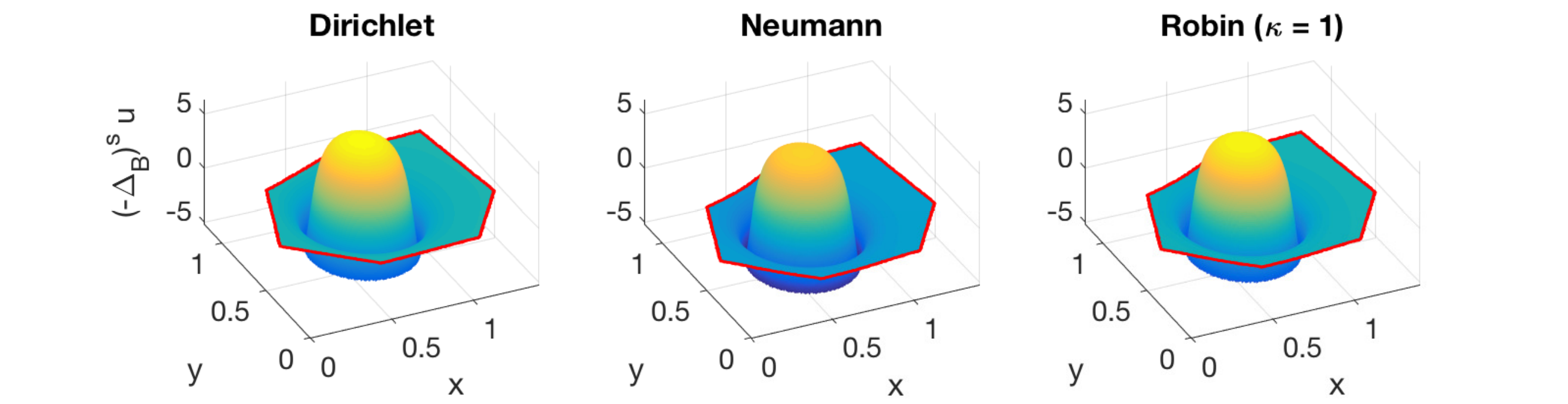}
	\includegraphics[width=\textwidth]{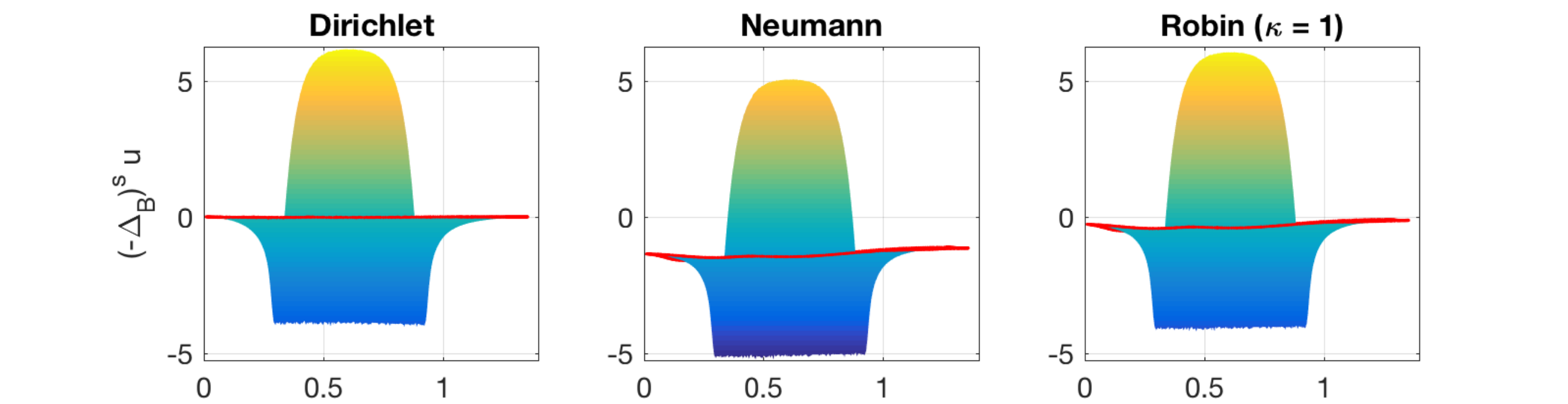}
	\includegraphics[width=\textwidth]{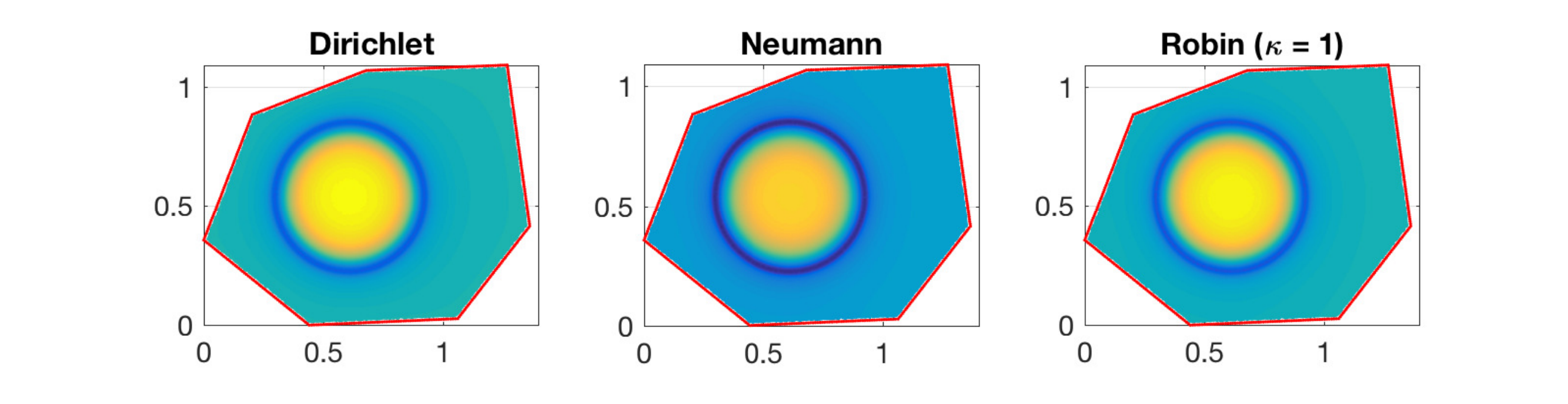}
	\caption{Fractional Laplacian ($s=0.5$) with different boundary conditions applied to the same function on convex polygonal domain. Top row: view with an angle. Middle row: side view. Bottom row: top view. Boundary nodes are highlighted in red to help with visualization. Parameters of the experiment: $r=0.35$, $x_0=0.61$, $y_0=0.54$, $k=1$, $\theta=1/2$, $p=1$, $\eta=0.001$}.
	\label{fig:Comp2DIrregDom}
\end{figure}	

Once again the value of the $\B$-fractional Laplacian with $s=0.5$ is computed for the compactly supported function $$u(x,y)=e^{-((x-x_0)^2+(y-y_0)^2)/(r^2-((x-x_0)^2+(y-y_0)^2))}\mathbbm{1}_{\{((x-x0)^2+(y-y0)^2)<r^2\}}$$
with homogeneous Dirichlet, Neumann, and Robin boundary constraints. Different views of $(-\Delta_\B)^su$ are reported to help in visualizing the differences (especially in the behavior close to the boundary) between these three cases. 
	
\subsection{An explicit scheme for a fractional porous medium type equation}	
As a possible application of our method, we compute the numerical approximation of an evolution problem involving the homogeneous Dirichlet fractional Laplacian. Specifically, on a given bounded domain $\Omega$ and for $\tau \geq 0$, we consider the fractional porous medium type equation 
\begin{equation} \label{eq:pormed}
\partial_\tau u+(-\Delta_\B)^s (u^m)=0,
\end{equation}
where $m$ is a suitable integer, $m>1$. For a precise presentation of the theoretical aspects regarding \eqref{eq:pormed} we refer to \cite{bonforte_etal-arxiv-2017,BoSiVa15} and references therein.

Given an initial condition $u(x,0)=u^{(0)}(x)$, we use an explicit discretization scheme for the evolution variable $\tau$ with uniform step $\varDelta \tau>0$ and compute iteratively the numerical approximation of the solution to \eqref{eq:pormed} at $\tau_n= (n+1) \varDelta \tau$, $n\in \N$ as
\begin{equation}\label{eq:numpormed}
u^{(n+1)}=\varDelta \tau \left[ u^{(n)}+\Theta_h^s [(u^{(n)})^m] \right].
\end{equation}
As \eqref{eq:numpormed} is an explicit scheme, it may suffer from numerical instabilities when the time step $\varDelta \tau$ is not properly chosen. We circumvent the issue by imposing on $\varDelta \tau$ a CFL-type condition and setting $\varDelta \tau= h^{2s}/m$. More details about these facts can be found in \cite{dTEnJa18} dealing with \eqref{eq:pormed} in $\R^N$.

In \cref{fig:FPM_sim} we report the results obtained for two different combinations of the parameters $m$ and $s$, at the indicated time point $\tau$. In both cases, the spatial domain is $\Omega=(-1,1)$, while the initial condition (plot on the left in \cref{fig:FPM_sim}) is the compactly supported function $u^{(0)}(x)=e^{4-1/[(0.5-x)(0.5+x)]}\mathbbm{1}_{|x|<0.5}$. 
\begin{figure}
	\centering
	\subfigure{\includegraphics[width=0.3\textwidth]{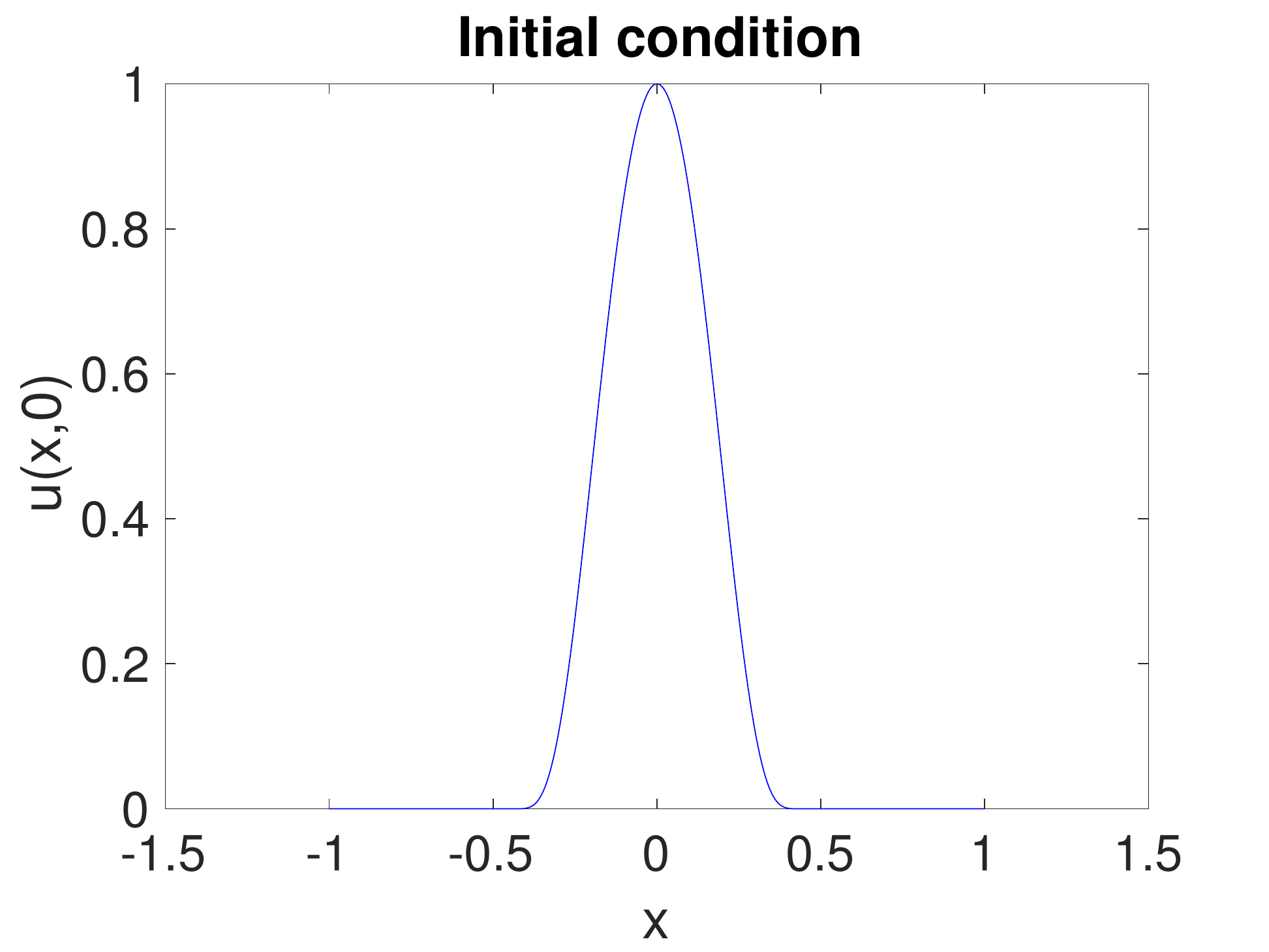}}
	\subfigure{\includegraphics[width=0.3\textwidth]{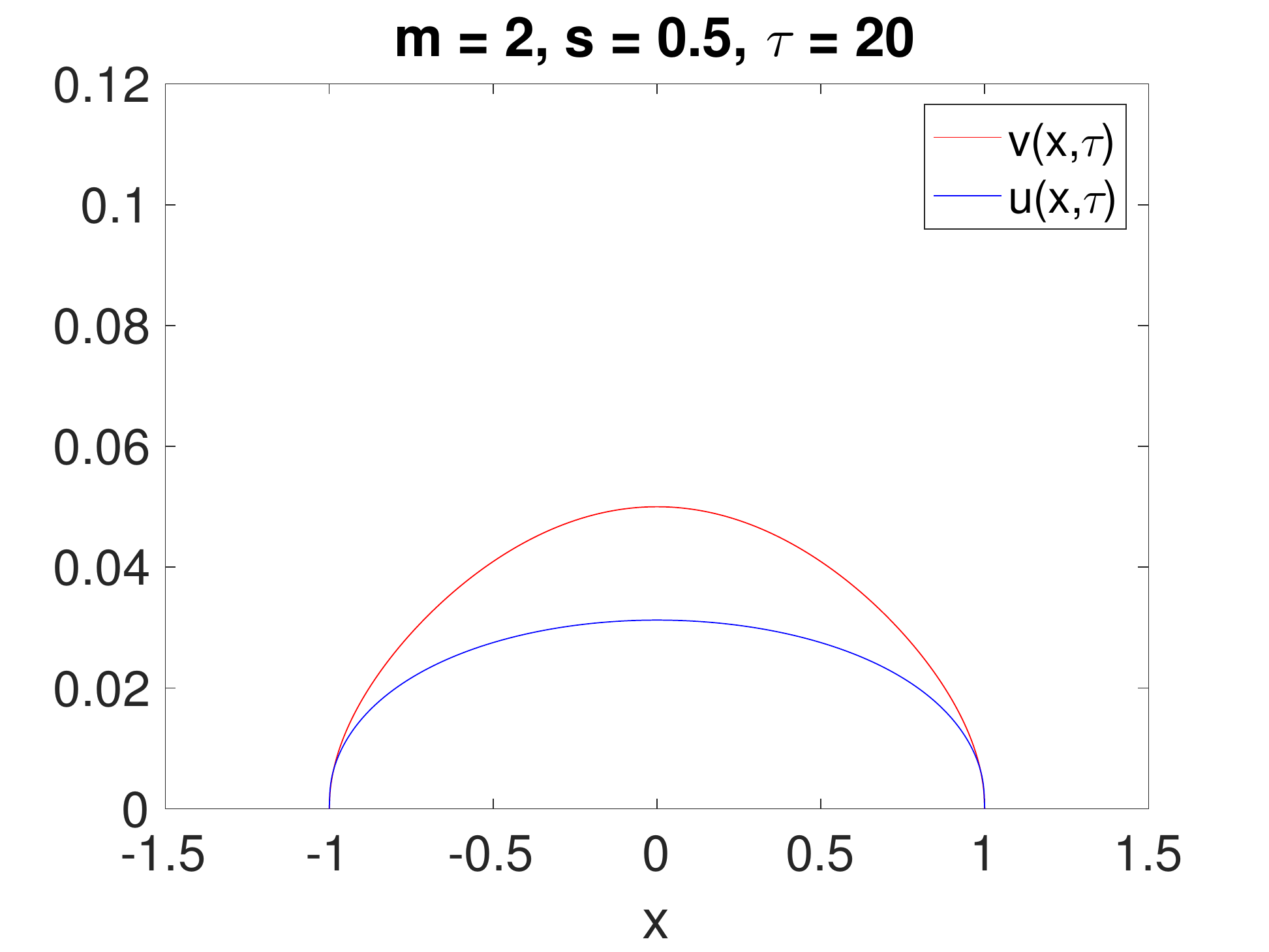}}
	\subfigure{\includegraphics[width=0.3\textwidth]{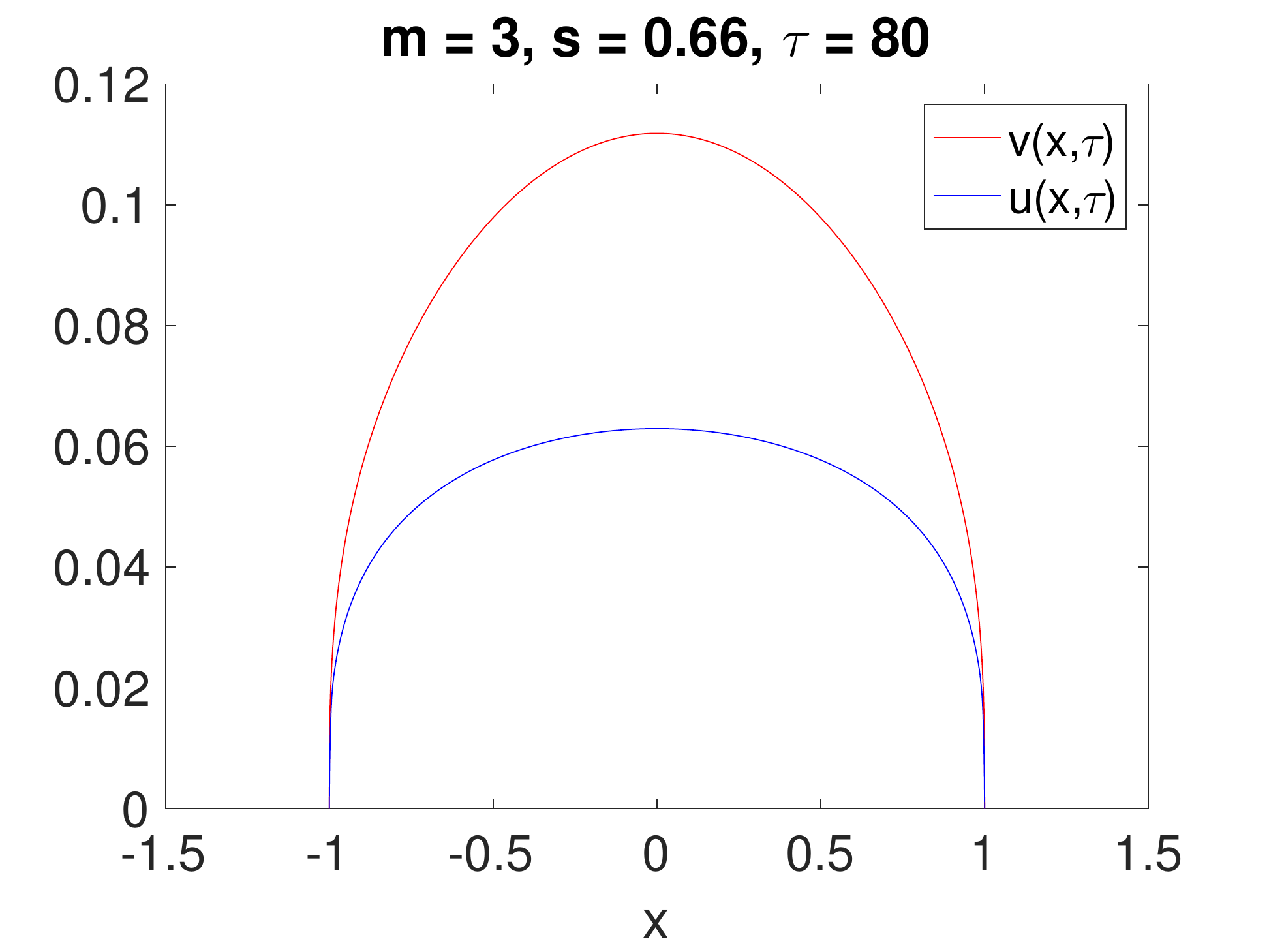}}
	\caption{Fractional porous medium initial condition, numerical solution (blue) and scaled eigenfunction determining the boundary behavior (red). In all cases $\Omega=(-1,1)$, $N_h=1001$ ($h=0.002$), and $\varDelta \tau=h^{2s}/m$. Values of $m$, $s$, and $\tau$ as indicated in the respective plots. Parameters of the experiment: $k=1$, $\theta=1$, $p=1$, $\eta=1$.}
	\label{fig:FPM_sim}
\end{figure}	

For the particular combinations of $m$ and $s$ presented in \cref{fig:FPM_sim} the theoretical results of Bonforte et al.~\cite{bonforte_etal-arxiv-2017} predict a long-time boundary behavior of the type
\[
u(x,\tau)\asymp v(x,\tau):=\frac{\phi_1( x)^{1/m}}{\tau^{1/(m-1)}},
\]
where $\phi_1$ is the first eigenfunction of the Laplacian with homogeneous Dirichlet boundary condition in $\Omega$. Given $a,b \in \R_+$ we follow the notation used in \cite{bonforte_etal-arxiv-2017} and  write $a\asymp b$ whenever there exist universal constants $c_0,c_1>0$ such that $c_0 b\leq a \leq c_1 b$.

The predicted boundary behaviour presented in \cite{bonforte_etal-arxiv-2017}  is recovered by our simulations and shown in the second and third plots of \cref{fig:FPM_sim}, where the numerical solution obtained with our method (blue) is compared to the expected function $v(x,\tau)$ (red).

\section{Conclusions}
\label{sec:conclusions}
In the present paper we investigate a method to define a nonlocal operator acting on bounded domains that, in the unbounded case, recovers the classical definition of the fractional Laplacian $(-\Delta)^s$. The studied method is characterized by the fact that the imposition of local boundary conditions is sufficient to obtain a well-posed formulation of the nonlocal operator without requiring additional constraints to be imposed on the complement of the considered finite domain, in analogy with local problems. As pointed out in the introduction, this characteristic is relevant for applied problems and for the physical foundation of fractional models.

We consider for our purposes the spectral definition of the fractional Laplacian and exploit the equivalence between this definition and the heat-semigroup formulation of the operator on bounded domains. The latter definition allows for a much more intuitive and transparent formulation in which the role played by the local boundary conditions (passed on to the local heat problem) is clear and all ambiguity of having local conditions for a nonlocal operator is removed.
 
This results in a power-law weighted integration of the evolving discrepancy between the solution of the heat equation and its initial condition. Within this framework, we propose a novel discretization scheme. An FE method is adopted to obtain the approximated solution of the heat equation and suitable quadrature formulas are used to approximate the integral.

Unlike other discretization methods, the proposed approach does not rely on the computation (or the approximation) of the spectrum of the Laplacian operator on the considered finite domain. In fact, a key feature of our strategy lies in the flexibility offered by FE in handling well not only simple regular geometries but also general bounded domains for which explicit knowledge of the Laplacian spectrum is missing. Moreover, once the accurate computation of the heat equation solution is available, our method simply requires the approximation of a one-dimensional integral, independently from the dimension in which the fractional Laplacian of the given function is sought.

In this work, we start by considering the case of homogeneous boundary conditions and propose two discretizations of the fractional operator (depending on a discretization parameter $h$) exhibiting two different rates of convergence. The main difference between the two results lies in the quadrature approximation of the power-law weighted integral. In the first discretization, we truncate the integral by ignoring its singular part (near the origin), adopt a na\"{i}ve mid-point quadrature rule for the middle part, and exploit the exponential convergence to the steady-state of the solution of the heat equation to approximate the tail. Doing so and suitably choosing the parameters of the FE computation, we obtain an error that behaves like $\mathcal{O}(h^{p(1-s)})$ for some positive $p$ and for values of the fractional order $s\in (0,1)$. This approach produces satisfactory results for values of $s$ close to its lower bound. 
By adapting the quadrature of the integral representing the operator and ensuring that the FE provide the required higher temporal accuracy for the solution to the heat equation, we derive a second discretization resulting in an error which is $\mathcal{O}(h^{p(2-s)})$. Although the proof of this result comes at the expenses of requiring some higher regularity to the function on which the action of the fractional Laplacian is computed, the accuracy of this second discretization ensures a decay of the error at least of the order $\mathcal{O}(h^{p})$, and hence this discretization is suitable to deal with the case of $s$ close to one. 

The case of non-homogeneous boundary conditions is considered next. We propose a definition of a fractional-type operator consistent with these inhomogeneous constraints and formulated as the power-law weighted integral of the evolving discrepancy of a heat equation solution, in which non-homogeneous boundary conditions are imposed, and its initial condition satisfying these conditions. We show that such an operator can be rewritten as the fractional Laplacian coupled to homogeneous boundary conditions, for a suitable shifted function, and we use the same two discretizations obtained in the homogeneous case to derive corresponding discretizations of the new nonlocal operator.  

Our numerical results show performance of our method for both one and two dimensional domains. The theoretical rate of convergence is very well recovered in all the analyzed cases and in the two-dimensional setting we also illustrate the performance of our method on a general irregular bounded domain. We conclude the numerical section with an example of a parabolic problem for which the computation of the fractional Laplacian could be used. In particular, we consider a fractional porous medium equation in which homogeneous Dirichlet conditions are imposed at the boundary. We compare our numerical solution with the analytic results on the boundary behavior of the solution given in \cite{bonforte_etal-arxiv-2017} and show that the expected behavior is well captured by our method, for both combinations of parameters considered.  
 
The approach followed in this work provides a natural framework that can be used to obtain either discretizations of the fractional Laplacian defined via the heat-semigroup on the whole $\mathbb{R}^n$ or to compute discretizations of inverse fractional powers of second order operators (i.e. solutions of elliptic problems associated to fractional operators) on bounded domains. In the unbounded settings, the use of FE to accurately model non rectangular domains is not required anymore due to the absence of a boundary, and the simpler finite difference approach can be adopted. On bounded domains, inverse fractional powers of second order operators can still be represented via integral expressions involving the corresponding heat-semigroups. The kernel of these expressions become locally integrable near the origin, avoiding some of the difficulties we had to deal with in the present work, but integrability is lost at infinity. Nevertheless, once again one can exploit the exponential convergence of the solution to the heat equation towards a steady-state to overcome this issue. Both these generalizations are current research topics of the authors of this paper, and the outlined results will be presented in two forthcoming papers. 
 
\section*{Acknowledgments} 
We are grateful to the two anonymous reviewers of this paper for their careful revision and useful comments that helped to improve significantly the quality of our manuscript. This research is supported by the Basque Government
through the BERC 2014-2017 program and by the Spanish Ministry of Economy and Competitiveness MINECO through BCAM Severo Ochoa excellence accreditation SEV-2013-0323, and through projects MTM2015-69992-R ``BELEMET''; MTM2016-76016-R ``MIP''. F.~del Teso is also supported by the Toppforsk (research excellence) project Waves and Nonlinear Phenomena (WaNP), grant no.~250070 from the Research Council of Norway, and the ERCIM ``Alain Benoussan" fellowship program. The work was done during the second author's stay at BCAM as a visiting fellow.

\bibliographystyle{siamplain}

\end{document}